\documentclass[10pt,reqno,a4paper]{amsart} 

\usepackage{graphicx}
\usepackage{amssymb,amsmath}
\usepackage{amsthm}
\usepackage{color,graphicx}
\usepackage{hyperref}
\usepackage{color}
\usepackage{mathabx}
\usepackage{subfigure}
\usepackage{enumerate}

\theoremstyle{definition} 
\usepackage{xcolor}

\usepackage{subfigure} 

\usepackage{verbatim} 

\usepackage{relsize} 

\newcommand{\be}{\begin{equation}}

\newcommand{\ee}{\end{equation}}

\newcommand{\cn}{{\rm \,cn}}
\newcommand{\sn}{{\rm \,sn}}
\newcommand{\dn}{{\rm \,dn}}

\newcommand{\K}{{\rm \,K}}
\newcommand{\E}{{\rm \,E}}

\numberwithin{equation}{section}
\numberwithin{figure}{section}

\newtheorem{theorem}{Theorem}[section]
\newtheorem{proposition}[theorem]{Proposition}
\newtheorem{remark}[theorem]{Remark}
\newtheorem{lemma}[theorem]{Lemma}

\newtheorem{definition}[theorem]{Definition}

\begin{document}
\vglue-1cm \hskip1cm
\title[Multiple periodic waves of the Schrödinger system with quintic nonlinearity]{Schrödinger system with quintic nonlinearity: \\ spectral stability of multiple sign-changing periodic waves
\\ \vspace{0.2cm}}

\begin{center}



\maketitle

{\bf Gabriel E. Bittencourt Moraes}

{Departamento de Matem\'atica - Universidade Estadual de Londrina \\
Londrina, PR, Brazil.}\\
{gbittencourt@uel.br}

\vspace{0.5cm}

{\bf Guilherme de Loreno}

{Departamento de Matem\'atica - Universidade Estadual do Centro-Oeste \\ 
Guarapuava, PR, Brazil.}\\
{guilherme.loreno@unicentro.br}

\vspace{3mm}

\end{center}

\begin{abstract}
This manuscript investigates the existence and spectral stability of multiple periodic standing wave solutions for a nonlinear Schrödinger system. By considering both \textit{cnoidal} and \textit{snoidal} profiles, we provide a comprehensive spectral analysis of the associated linearized operators, employing the Floquet theory and comparison theorems. Stability results are derived under periodic perturbations with the same period as the underlying standing waves. Furthermore, we apply the spectral stability theory via Krein signature, as developed in \cite{KapitulaKevrekidisSandstedeI} and \cite{KapitulaKevrekidisSandstedeII}, to determine the spectral stability and instability results.
\end{abstract}

\section{Introduction} 

This paper investigates the spectral stability of periodic standing waves for the quintic nonlinear Schrödinger (NLS) system
\begin{equation}\label{NLS-system}
	\left\{ \begin{array}{l}
		i u_t + u_{xx} + \kappa |u|^4 u + \gamma v^3 \overline{u}^2 = 0 \\
		i v_t + v_{xx} + \kappa |v|^4 v + \gamma u^3 \overline{v}^2 = 0,
	\end{array}
\right.
\end{equation}
where $u,v: \mathbb{R} \times (0,+\infty) \rightarrow \mathbb{C}$ are complex-valued functions that are $L$-periodic in the spatial variable for $L>0$. The system parameters $\kappa > 0$ and $\gamma \geq 0$ are real constants. In this work, we are particularly interested in the stability of periodic solutions that change sign over the period $[0, L]$, commonly referred to as periodic sign-changing solutions.

The scalar nonlinear Schrödinger equation (obtained by setting $v = 0$ in \eqref{NLS-system}) arises in various physical and biological contexts. These include nonlinear optics, Bose-Einstein condensates, and the description of nonlinear waves such as laser beam propagation, water waves at the free surface of an ideal fluid, and plasma waves. Furthermore, it appears in DNA modeling and mesoscopic molecular structures. While the cubic NLS system has been extensively studied in these fields (see \cite{Agrawal, IedaMiyakawaWadati, KannaSakkaravarthi, WangCui}), the quintic case has recently gained attention due to its relevance in describing higher-order nonlinear effects in optical fibers and dilute gases.

Regarding the scalar quintic NLS equation in the periodic setting, several studies have established orbital and spectral stability results for positive solutions with dnoidal profiles (see \cite{AnguloNatali2009, HakkaevStanislavovaStefanov}). For sign-changing solutions with cnoidal profiles, Natali \textit{et al.} in \cite{NataliMoraesLorenoPastor} proved the orbital stability of cnoidal waves for the cubic NLS equation. Specifically, they demonstrated the existence of a frequency threshold $\omega^*>0$ such that the cnoidal standing wave is orbitally stable within the subspace of $H^1_{\rm per}$ consisting of mean-zero functions. In the quintic case, Moraes \& Loreno in \cite{BittencourtLoreno2022} addressed the orbital stability of cnoidal waves within the subspace of even functions in $H^1_{\rm per}$, following the approach in \cite{grillakis2, ShatahStrauss}. However, we have identified a subtle inconsistency in the application of the Vakhitov-Kolokolov stability criterion in that work. Consequently, the previous stability conclusions require revision, a task we undertake in Section \ref{section-orbital}.

The literature on NLS systems has also seen significant progress. Pastor in \cite{Pastor} studied a cubic NLS system with specific coupling constants, proving orbital stability for dnoidal standing waves under periodic perturbations of the same period, based on the framework in \cite{grillakis1, grillakis2}. Additionally, by employing the theories from \cite{GallayHaragusJDE, grillakis1}, the author established spectral stability results for periodic waves concerning localized and bounded perturbations.

Similarly, Hakkaev in \cite{Hakkaev2019} considered a cubic NLS system of the form
\begin{equation}\label{NLS-system-cubic}
	\left\{ \begin{array}{l}
		i u_t + u_{xx} + \kappa_1 |u|^2 u + \gamma v^2 \overline{u} = 0 \\
		i v_t + v_{xx} + \kappa_2 |v|^2 v + \gamma u^2 \overline{v} = 0
	\end{array}
\right.
\end{equation}
where $u, v : \mathbb{R} \times (0, +\infty) \to \mathbb{C}$,  $\kappa_1,\kappa_2>0$ and $\gamma \geq 0$. The author investigated the spectral stability of semi-trivial standing waves of the form $(u,v) = (\varphi,0)$ with a dnoidal profile, showing they are orbitally stable for $\gamma < \kappa_1$. Furthermore, the study determined that for $\kappa_1 < \gamma \leq 3\kappa_1$, these semi-trivial waves become spectrally unstable, while for $\gamma = \kappa_1$, stability is maintained. 

In a related direction, Natali \& Moraes in \cite{NataliBittencourt2024} examined the existence and spectral stability of multiple periodic waves for \eqref{NLS-system-cubic} of the form \((u(x,t), v(x,t)) = \left( e^{i \omega t} \varphi(x), e^{i \omega t} B \varphi(x) \right)\), where $B>0$. Their analysis utilized the Krein signature theory as developed in \cite{KapitulaKevrekidisSandstedeI, KapitulaKevrekidisSandstedeII} by Kapitula, Kevrekidis \&  Sandstede, with results depending on the parameters $\kappa_1$, $\kappa_2$, and $\gamma$. Furthermore, the authors established spectral stability results for solutions with dnoidal profiles in $H^1_{\rm per}$, as well as for cnoidal profiles within both $H^1_{\rm per}$ and the subspace of odd functions. It is worth noting that the technical challenges found by the authors regarding cnoidal waves in the full $H^1_{\rm per}$ space motivated the present study of the quintic case.

To the best of our knowledge, investigations concerning the spectral and orbital stability of the quintic Schrödinger system \eqref{NLS-system} remain largely underexplored in the existing literature. This gap constitutes the main motivation for the present work. Specifically, we aim to provide a detailed analysis of the spectral stability properties for system \eqref{NLS-system}, clarifying several fundamental points that have not yet been fully addressed, particularly regarding the role of sign-changing configurations.

Motivated by the considerations above, we outline the primary objectives of this paper. First, we consider multiple periodic standing wave solutions for \eqref{NLS-system} of the form
\begin{equation}\label{multiplesol}
    (u(x,t), v(x,t)) = e^{i\omega t}(\varphi(x), B \varphi(x)), \quad (x,t) \in \mathbb{R} \times \mathbb{R}_+
\end{equation}
where $\varphi: \mathbb{R} \rightarrow \mathbb{R}$ is an $L$-periodic function, $L>0$,  $\omega \in \mathbb{R}$ represents the wave frequency, and $B \in \mathbb{R}$ is a constant. We should note that in \cite{NataliBittencourt2024}, the parameter $B$ could be restricted to non-negative values due to the reflection symmetry $v \mapsto -v$ present in the cubic system \eqref{NLS-system-cubic}. In contrast, our quintic system \eqref{NLS-system} does not possess this particular symmetry; hence, we must consider the more general case where $B$ can be any real number.

In this paper, we focus on periodic sign-changing waves $\varphi$ with a cnoidal profile -- a case that presented significant technical challenges in \cite{NataliBittencourt2024}. Subsequently, since the system \eqref{NLS-system} in the periodic setting is invariant under spatial translations,   by applying a spatial translation of $\frac{L}{4}$ to the right, we obtain a profile of snoidal type. This translated solution is odd-symmetric, a property that allows us to overcome several difficulties encountered when analyzing cnoidal waves over the whole space $H^1_{\rm per}$. To construct these periodic solutions, we substitute the ansatz \eqref{multiplesol} into \eqref{NLS-system}, which yields the following coupled system:
\begin{equation}\label{coupledsystem}
\begin{cases}
-\varphi'' + \omega \varphi - (\kappa + \gamma B^3)\varphi^5 = 0 \\
-\varphi'' + \omega \varphi - (\kappa B^4 + \gamma B \varphi^5) = 0. 
\end{cases}
\end{equation}
To guarantee the existence of such standing wave solutions, we impose the following condition
\begin{equation}\label{Constantsassumptions}
    \kappa + \gamma B^3 = \kappa B^4 + \gamma B.
\end{equation}
which is an algebraic equation for the unknown $B \in \mathbb{R}$. The solutions to this equation are given by
\begin{equation}\label{Bvalue}
B = B_+: = \frac{\gamma + \sqrt{\gamma^2 - 4\kappa^2}}{2\kappa}, \quad B = B_-: = \frac{\gamma - \sqrt{\gamma^2 - 4\kappa^2}}{2\kappa}, \quad B = 1, \quad \text{and} \quad B = -1.
\end{equation}
Based on these four possible values for $B$, our analysis is naturally divided into four distinct cases. Consequently, under the condition \eqref{Constantsassumptions}, in view of \eqref{coupledsystem}, the periodic function $\varphi$ satisfies the following quintic scalar ordinary differential equation:
\begin{equation}\label{mainEDOquinticsystemIntroduction}
     -\varphi'' + \omega \varphi - (\kappa + \gamma B^3)\varphi^5 = 0.
\end{equation}

An essential mathematical aspect of the ordinary differential equation \eqref{mainEDOquinticsystemIntroduction} is the existence of $L$-periodic solutions with a cnoidal profile. Indeed, based on the results in \cite{BittencourtLoreno2022} and \cite{NataliCardosoAmaral}, we can guarantee that the function
\begin{equation*}\label{cnoidalsolutionIntroduction}
\varphi(x) = \theta \frac{a \, \text{cn} \left( \frac{4{\rm K}(k)}{L} x, k \right)}{\sqrt{1 - q \, \text{sn}^2 \left( \frac{4{\rm K}(k)}{L} x, k \right)}}, \quad x \in \mathbb{R}
\end{equation*}
is a solution to \eqref{mainEDOquinticsystemIntroduction}, where $k \in (0,1)$ is the elliptic modulus and ${\rm K}(k)$ denotes the complete elliptic integral of the first kind. The parameters $a$ and $q$ depend smoothly on $k \in (0,1)$ and $L>0$, while the frequency $\omega>0$ is given by
\begin{equation}\label{valuew}
\omega = \frac{16 {\rm K}(k)^2 \sqrt{k^4 - k^2 + 1}}{L^2},
\end{equation}
with the scaling factor defined as $\theta := (\kappa + \gamma B^3)^{-1/4}$.

Furthermore, by defining the translated function $\psi(x) = \varphi\left( x + \frac{L}{4} \right)$, for all $x \in \mathbb{R}$, and employing some Jacobian identities from \cite{byrd}, we obtain an $L$-periodic odd solution to \eqref{mainEDOquinticsystemIntroduction} with a snoidal profile, expressed as
\begin{equation}\label{psisolutionintro}
\psi(x) = \theta \frac{ \tilde{a} \, \text{sn}\left( \frac{4{\rm K}(k)}{L}x,k\right)}{\sqrt{ 1 - \tilde{q} \, \text{sn}^2 \left( \frac{4 {\rm K}(k)}{L} x, k \right)}}, \quad x \in \mathbb{R}
\end{equation}
where the parameters $\tilde{a}$ and $\tilde{q}$ also depend smoothly on $k \in (0,1)$ and $L>0$, for the same frequency $\omega$ as in \eqref{valuew}.

\begin{remark}
    For the solutions introduced above, we can construct a smooth curve of $L$-periodic waves $\omega \in \big( \tfrac{4\pi^2}{L^2}, +\infty \big) \longmapsto \varphi_\omega, \psi_\omega \in H^2_{per}$ satisfying \eqref{mainEDOquinticsystemIntroduction} for each considered case (see Propositions \ref{cnoidalcurve} and \ref{snoidalcurve}).
\end{remark}

In another important direction, it is worth noting that system \eqref{NLS-system} admits the conserved quantities $E, F: H^1_{\text{per}} \times H^1_{\text{per}} \to \mathbb{R}$, defined as
\begin{equation}\label{EnergyFunctional}
E(u,v) = \frac{1}{2} \int_0^L \left( |u_x|^2 + |v_x|^2 - \frac{\kappa}{3} \left( |u|^6 + |v|^6 \right) - \gamma \text{Re}(u^3 \overline{v}^3) \right)\; dx
\end{equation}
and
\begin{equation}\label{MassFunctional}
F(u,v) = \frac{1}{2} \int_0^L \left( |u|^2 + |v|^2 \right) \; dx,
\end{equation}
which represent the energy and mass of the system, respectively. By employing the energy functional \eqref{EnergyFunctional} and following the arguments in \cite{Cazenave}, one can show that the Schrödinger system \eqref{NLS-system} is locally well-posed in the energy space $H^1_{\text{per}} \times H^1_{\text{per}}$. More formally, we have the following result:

\begin{theorem}[Local Well-Posedness]
The Cauchy problem associated with system \eqref{NLS-system} with initial data $U_0 = (u_0, v_0) \in H_{\text{per}}^1 \times H_{\text{per}}^1$ is locally well-posed. More precisely, for each $U_0 \in H_{\text{per}}^1 \times H_{\text{per}}^1$, there exists a time $T > 0$ and a unique solution $U = (u, v) \in C([0,T]; H_{\text{per}}^1 \times H_{\text{per}}^1)$ such that $U(0) = U_0$. Moreover, for each $T_0 \in (0,T)$, the mapping
\[
U_0 \in H_{\text{per}}^1 \times H_{\text{per}}^1 \longmapsto U \in C([0,T_0]; H_{\text{per}}^1 \times H_{\text{per}}^1)
\]
is continuous.
\end{theorem}
 
Next, we describe the framework for establishing the spectral stability of synchronized periodic waves with respect to perturbations of the same period $L>0$. For clarity, we identify the complex-valued solution $U = (u,v)$ of system \eqref{NLS-system} with its real and imaginary parts as $U = (\text{Re } u, \text{Re } v, \text{Im } u, \text{Im } v) \in \mathbb{R}^4$. 

In this setting, the stationary wave is represented by $\Phi = (\varphi, B\varphi, 0, 0)$. We consider a linear perturbation of the form
\begin{equation}\label{U-1}
	U(x,t) = e^{i \omega t} (\Phi(x) + W(x,t)), \quad (x,t) \in \mathbb{R} \times \mathbb{R}_+
\end{equation} 
where $W(x,t) = (\text{Re } w_1, \text{Re } w_2, \text{Im } w_1, \text{Im } w_2)$ is a small perturbation in $\mathbb{R}^4$. Substituting \eqref{U-1} into \eqref{NLS-system} and linearizing around $\Phi$ by neglecting higher-order terms, we obtain the linearized evolution equation:
\begin{equation}\label{spectral}
	\frac{d}{dt} W(x,t) = J \mathcal{L} W(x,t), \quad (x,t) \in \mathbb{R} \times \mathbb{R}_+.
\end{equation}
In this expression, $J$ denotes the standard $4 \times 4$ matrix
\begin{equation}\label{Jmatrix}
J = \begin{bmatrix}
0 & \text{Id}_{2 \times 2} \\
		-\text{Id}_{2 \times 2} & 0    
\end{bmatrix},
\end{equation}
 and $\mathcal{L}$ is the self-adjoint linearized operator defined by
\begin{equation}\label{L-1}
	\mathcal{L} = (-\partial_x^2 + \omega) \text{Id}_{4 \times 4} - \varphi^4 S,
\end{equation}
where $S$ is a constant matrix depending on the system parameters $\kappa, \gamma,$ and $B$ given as
\begin{equation}\label{Smatrix}
    S = \begin{bmatrix}
5\kappa + 2\gamma B^3 & 3\gamma B^2 & 0 & 0 \\
3\gamma B^2 & 5\kappa B^4 + 2\gamma B & 0 & 0 \\
0 & 0 & \kappa - 2\gamma B^3 & 3\gamma B^2 \\
0 & 0 & 3\gamma B^2 & \kappa B^4 - 2\gamma B
\end{bmatrix}.
\end{equation}
Here, given $n \in \mathbb{N}$, we denote by $\text{Id}_{n \times n}$ is the $ n\times n$ identity matrix.

To analyze the spectral properties of the system, we assume the separation of variables $W(x,t) = e^{\lambda t}w(x)$, which leads to the eigenvalue problem:
\begin{equation*}
	J \mathcal{L} w = \lambda w.
\end{equation*} 

The spectral problem above enables us to introduce the central concept that motivated our study (as introduced in a general context in \cite[Section 5.3]{AnguloBook}). The stability of the periodic wave is then characterized by the spectrum of the operator $J \mathcal{L}$, denoted by $\sigma(J \mathcal{L})$, as follows.

\begin{definition}[Spectral Stability]\label{Spectral Stability Definition}
We say that the stationary wave \(\Phi\) is \textit{spectrally stable}  with respect to periodic perturbations having the same period $L>0$  if \(\sigma(J\mathcal{L}) \subset i\mathbb{R}\). Otherwise, if there exists at least one eigenvalue \(\lambda \in \mathbb{C}\) associated with the operator \(J\mathcal{L}\) that has a positive real part, we say that \(\Phi\) is \textit{spectrally unstable}.
\end{definition}


Next, with the aim of clarifying our approach, given a linear operator $\mathcal{T}$ defined in a subspace of a Hilbert space, denote by ${\rm n}(\mathcal{T}), {\rm z}(\mathcal{T}) \in \mathbb{N}$ the number of negative eigenvalues of $\mathcal{T}$ and the dimension of the kernel of $\mathcal{T}$.

Our main tool for analyzing spectral stability or instability will be the classical framework developed by Kapitula, Kevrekidis \&  Sandstede in \cite{KapitulaKevrekidisSandstedeI} and \cite{KapitulaKevrekidisSandstedeII}. This approach is particularly powerful for determining the spectral properties of abstract Hamiltonian systems. Specifically, when ${\rm z}(\mathcal{L}) = n \in \mathbb{N}$ and the set $\{ \Theta_i \in {\rm D}(\mathcal{L}) \; ; \; i \in \{1,\ldots,n\}\} \subset {\rm Ker}(\mathcal{L})$ is linearly independent, their method allows us to construct a matrix $V \in \mathbf{M}_{n \times n}(\mathbb{R})$ with entries defined by
\begin{equation}\label{Ventries}
V_{ij}=( \mathcal{L}^{-1} J \Theta_i, J \Theta_j )_{L^2 \times L^2},
\end{equation}
for $i,j \in \{1,2,\cdots,n\}$ and provided that $V$ is invertible,  the following formula is valid:
\begin{equation}\label{KreinIndex}
 \mathcal{K_{\rm Ham}}=\text{n}(\mathcal{L}) - \text{n}(V) 
\end{equation}
where $ \mathcal{K_{\rm Ham}}:=k_r + k_c + k_{-}$ is called the Hamiltonian--Krein Index (\cite[Definition 7.1.4]{KapitulaPromislow}), which is a very useful tool to obtain spectral stability and instability of periodic waves. The formula \eqref{KreinIndex} can also be rigorously established by means of the result in \cite[Theorem 7.1.5]{KapitulaPromislow}, known as the Hamiltonian--Krein Index Theorem. Here, $k_r, k_c, k_{-} \in \mathbb{N}$ are such that:
\begin{itemize}
    \item $k_r$ denotes the number of positive real eigenvalues of the operator $\mathcal{L}$ (counting multiplicities);
    \item $k_c$ denotes the number of complex eigenvalues with positive real part of the operator $\mathcal{L}$;
    \item $k_{-}$ denotes the number of purely imaginary eigenvalues with negative Krein signature of the operator $\mathcal{L}$.
\end{itemize}
We note that the values $k_c$ and $k_{-}$ are always even numbers. It follows that if $\mathcal{K_{\rm Ham}}$ is odd, then $k_r \geq 1$, and consequently, we obtain a result of spectral instability. Moreover, if $\mathcal{K_{\rm Ham}} = 0$, then $k_r = k_c = k_{-} = 0$, and in this case, we obtain a result of spectral stability. Based on these conclusions, we will later be able to establish our results on spectral stability and instability (see Theorems \ref{teorema-cnoidal} and \ref{teorema-snoidal}).

Bearing in mind the above, we see that a study of the spectrum of the operator $\mathcal{L}$ is essential to determine a result of spectral stability or instability. However, since $\mathcal{L}$ is not a diagonal operator, this analysis can be rather difficult or complicated. To overcome this difficulty, we need to study the spectrum of $\mathcal{L}$ by means of a diagonal operator, that is, we need to obtain a similarity transformation of the operator $\mathcal{L}$ using a diagonal operator. This means that we need to find a real orthogonal matrix\footnote{An orthogonal matrix is a square matrix $A$ that is invertible such that its inverse $A^{-1}$ coincides with its transpose $A^T$. For more details about orthogonal matrices and their main properties see \cite[Section 4.6]{Williams}.} 
 $U\in \mathbf{M}_{4 \times 4}(\mathbb{R})$ and a diagonal matrix $M\in \mathbf{M}_{4 \times 4}(\mathbb{R})$ such that
\begin{equation}\label{similaritymatrixequation}
S= U  M  U^{-1}.   
\end{equation}
We should note that, in this case, $U$ is a orthogonal matrix satisfying $U^{-1}=U^T=U^*$, since $U$ is a real matrix and $U^*$ denotes de adjoint matrix\footnote{Adjoint or conjugate transpose as defined in \cite[Section 6.1]{Friedberg}.} of $U$. The equality \eqref{similaritymatrixequation} allows us to conclude that the operator $\mathcal{L}$ can be diagonalized
 under a similarity transformation as
\begin{equation}\label{auxiliarylinearizedoperatorrelation}
    \mathcal{L}= U \widetilde{\mathcal{L}} U^{-1},
\end{equation}
where
\begin{equation}\label{auxiliarylinearizedoperatordefinition}
 \widetilde{\mathcal{L}}= \begin{bmatrix}
\mathcal{L}_1 & 0 & 0 & 0 \nonumber\\
0 & \mathcal{L}_3 & 0 & 0 \nonumber\\
0 & 0 & \mathcal{L}_2 & 0 \nonumber\\
0 & 0 & 0 & \mathcal{L}_4
\end{bmatrix}
\end{equation}
and, for each $i \in \{1,2,3,4\}$, the operators $\mathcal{L}_i:H^2_{\rm per} \subset L^2_{\rm per} \longrightarrow L^2_{\rm per}$ are given by
\begin{equation}\label{HillOperators}
\mathcal{L}_i = -\partial_x^2 + \omega - \beta_i \varphi^4,
\end{equation}
in which $\beta_i \in \mathbb{R},\; i \in \{1,2,3,4\}$ are numbers to be determined. The operators introduced in \eqref{HillOperators}, commonly referred to as Hill operators are well established in the spectral theory of periodic differential operators and it is well known that they are closed, unbounded and self-adjoint with a closed range in $L^2_{\rm per}$ (see \cite[Section 7.2]{AnguloBook}). In addition, by \eqref{auxiliarylinearizedoperatorrelation} and using the orthogonality of the matrix $U$ and Sylvester’s Law of Inertia (\cite[Theorem 4.2]{Pelinovskybook}), we are able to study the number of zero and negative eigenvalues of the operator $\mathcal{L}$ by analyzing the number of zero and negative eigenvalues of the diagonal operator $\widetilde{\mathcal{L}}$. More precisely,
\begin{equation}\label{SylvesterInertialLaw}
    {\rm n}(\mathcal{L}) = {\rm n}(U \widetilde{\mathcal{L}} U^{-1}) = {\rm n}(\widetilde{\mathcal{L}}) = \sum_{i=1}^{4} {\rm n}(\mathcal{L}_i) \quad \text{and} \quad {\rm z}(\mathcal{L}) = {\rm z}(U \widetilde{\mathcal{L}} U^{-1}) = {\rm z}(\widetilde{\mathcal{L}}) = \sum_{i=1}^{4} {\rm z}(\mathcal{L}_i).
\end{equation}

Let us point out that, for each $B$ given in \eqref{Bvalue}, using standard linear algebra arguments, we are able to show the existence of the matrices $U$ and $M$ mentioned above such that
\begin{equation}\label{Mmatrix}
M = \begin{bmatrix}
\beta_1 & 0       & 0       & 0       \\
0       & \beta_3 & 0       & 0       \\
0       & 0       & \beta_2 & 0       \\
0       & 0       & 0       & \beta_4 \\
\end{bmatrix}
\end{equation}
where $\beta_i$, $i \in \{1,2,3,4\}$, are the real eigenvalues of the matrix $S$ expressed in terms of $\kappa$, $\gamma$ and $B$, and the matrix $U$ is constructed from a set of orthonormal eigenvectors associated with these eigenvalues.

Therefore, it should be emphasized that, in order to obtain the spectral analysis of the operator $\mathcal{L}$, it is enough to know the non-positive spectrum of the operators $\mathcal{L}_i$, for each $i \in \{1,2,3,4\}$. In this sense, we will next describe the results obtained in each case for each of the Hill operators. First, we observe that the spectral analysis of the operators $\mathcal{L}_i,\; i \in \{1,2,3,4\}$, depends on the wave $\varphi$ and the numbers $\beta_1, \beta_2, \beta_3$, and $\beta_4$. Next, based on the spectral theory developed in \cite{BittencourtLoreno2022}, we have precise information regarding the non-positive spectrum of $\mathcal{L}_1$ and $\mathcal{L}_2$. Specifically, for the cnoidal profile solution $\varphi$, for $\mathcal{L}_1$ we have  ${\rm n}(\mathcal{L}_1) = 2$ with $\text{Ker}(\mathcal{L}_1) = [\varphi']$, while for $\mathcal{L}_2$ we obtain ${\rm n}(\mathcal{L}_2) = 1$ with $\text{Ker}(\mathcal{L}_2) = [\varphi]$. To characterize the non-positive spectrum of $\mathcal{L}_3$ and $\mathcal{L}_4$, we employ a periodic version of the Comparison Theorem (\cite[Theorem 2.2.2]{Eastham}), which allows us to determine the distribution of negative eigenvalues through a chain of inequalities involving the operators $\mathcal{L}_i$, $i \in \{1,2,3,4\}$.

The high number of negative eigenvalues of $\mathcal{L}_1$ and $\mathcal{L}_2$ for the cnoidal profile presents significant technical challenges when applying the comparison theorem. To overcome this difficulty, we utilize the odd solution with snoidal profile $\psi$ and restrict the operator $\mathcal{L}$ to the subspace of odd functions, denoted by $\mathcal{L}_{\text{odd}}$. In this odd setting, we prove that ${\rm n}(\mathcal{L}_{1,\text{odd}}) = 1$ and ${\rm n}(\mathcal{L}_{2,\text{odd}}) = 0$. This spectral reduction enables a more direct application of the comparison theorem, leading to a more robust spectral stability result for the standing wave $\Psi = (\psi, B\psi,0,0)$ within the subspace of odd functions.

Another key element of our approach is the determination of ${\rm n}(V)$. This is achieved by analyzing the sign of $\det(V)$, using the fact that the determinant equals the product of the eigenvalues. A cornerstone of this analysis is deriving convenient and simplified expressions for the entries $V_{ij}$, for each $i,j \in \{1,2,\cdots, n\}$, of the matrix $V$, as defined in \eqref{Ventries}. To this end, by ensuring the orthogonality of the matrix $U$ (and in particular $U^* = U^T = U^{-1}$), we deduce from \eqref{auxiliarylinearizedoperatorrelation} the following fundamental relation:
\begin{equation}\label{adjointrelatioapplication}
    (\mathcal{L}^{-1} f, g)_{L^2 \times L^2} = (U\widetilde{\mathcal{L}}^{-1}U^{*} f, g)_{L^2 \times L^2} = (\widetilde{\mathcal{L}}^{-1}U^{-1} f, U^{-1}g)_{L^2 \times L^2},
\end{equation}
for all $f, g \in H^2_{\text{per}} \times H^2_{\text{per}}$. This identity is crucial for computing the entries of $V$. In addition, a crucial step in the analysis of $\det(V)$ is the study of the quantities
$
\mathcal{I} := \left( \mathcal{L}_1^{-1}(\varphi), \varphi \right)_{L^2}$
and $
\mathcal{J} := \left( \mathcal{L}_2^{-1}(\varphi'), \varphi' \right)_{L^2}.
$
To this end, we use the existence of a smooth curve
$
\omega \in \big( \tfrac{4\pi^2}{L^2}, +\infty \big) \longmapsto \varphi = \varphi_\omega
$
of periodic solutions, and show via algebraic arguments that $\mathcal{I} < 0$, thereby providing a slight improvement over the approach presented in \cite{BittencourtLoreno2022} (see Remarks \ref{rem:correction0} and \ref{rem:correction} for further details).
Moreover, by relying on standard Functional Analysis arguments, we are able to avoid numerical computations and prove that $\mathcal{J} > 0$.  Finally, our goal is to show that the difference between ${\rm n}(\mathcal{L})$ and ${\rm n}(V)$ is either zero or an odd integer. If this difference does not satisfy the stability criteria, that is if $\mathcal{K_{\rm Ham}}$ is a non-zero even number, then the spectral stability of the waves $\Phi = (\varphi, B\varphi, 0, 0)$ or $\Psi = (\psi, B\psi, 0, 0)$ remains inconclusive.

Summarizing the considerations above, our main result regarding the spectral instability of multiple periodic waves with a cnoidal profile is established as follows:

\begin{theorem}\label{teorema-cnoidal}
     Let $L>0$ be fixed and $\omega \in \left( \frac{4\pi^2}{L^2}, +\infty \right)$. Let the periodic wave solution $\varphi=\varphi_\omega \in H^1_{\text{per}}$ of \eqref{mainEDOquinticsystem} with the cnoidal profile given in \eqref{cnoidal-solution}. If $B = 1$ and $\gamma = 2\kappa$, then the standing wave $\Phi = (\varphi, \varphi, 0, 0)$ is spectrally unstable in $\mathbb{H}^1_{\text{per}}$.
\end{theorem}

Due to the high number of negative eigenvalues of $\mathcal{L}$ and the fact that the difference $\mathcal{K_{\rm Ham}}={\rm n}(\mathcal{L}) - {\rm n}(V)$ is even for most values of $B$ in \eqref{Bvalue}, we restrict our analysis to the subspace of odd functions. By considering the wave $\Psi = (\psi, B\psi, 0,0)$, we establish our main result concerning the spectral stability of multiple periodic waves with a snoidal profile:

\begin{theorem}\label{teorema-snoidal}
    Let $L>0$ be fixed and $\omega \in \left( \frac{4\pi^2}{L^2}, +\infty \right)$. Let the periodic solution $\psi=\psi_\omega \in H^1_{\text{per,odd}}$ of \eqref{mainEDOquinticsystem} with the snoidal profile given in \eqref{psisolutionintro}. Regarding the spectral stability of the synchronized periodic wave $\Psi = (\psi, B \psi, 0, 0) \in \mathbb{H}^1_{\text{per,odd}}$, the following assertions hold:
    \begin{enumerate}
        \item[(i)] If $B = \frac{\gamma \pm \sqrt{\gamma^2 - 4\kappa^2}}{2\kappa}$, then $\Psi$ is spectrally unstable.
        \item[(ii)] In the case $B = 1$, the wave $\Psi$ is spectrally unstable for $\gamma \in (0,2\kappa)$, and spectrally stable for $\gamma \in \{0\} \cup [2\kappa, +\infty)$.
        \item[(iii)] In the case $B = -1$, the wave $\Psi$ is spectrally stable provided that $\gamma = 0$.
    \end{enumerate}
\end{theorem}

The remainder of this paper is structured as follows. In Section \ref{SectionExistenceSolutions}, we demonstrate the existence of smooth curves of periodic standing wave solutions with cnoidal and snoidal profiles for equation \eqref{mainEDOquinticsystemIntroduction}. A detailed spectral investigation of the linearized operator $\mathcal{L}$ is also provided in Section \ref{SectionSpectralAnalysis}. The spectral stability and instability of multiple wave solutions with cnoidal (over $H^1_{\text{per}}$) and snoidal (over $H^1_{\text{per,odd}}$) profiles are discussed in Section \ref{Spectral stability Section} and \ref{spectral-stability-odd-space}, respectively. Finally, in Section \ref{section-orbital}, we provide remarks on a potential orbital stability result for the cnoidal standing wave solution within the subspace $H^1_{\text{per,odd}}$.

\vspace{0.5cm}

\noindent \textbf{Notation.} Let $s\geq0$ and $L>0$. Let $\mathcal{P} = C^\infty_{\rm per}$ denote the collection of all the functions $f : \mathbb{R} \to \mathbb{C}$ which are $C^\infty$ and periodic with period $L > 0$. The topological dual of $\mathcal{P}$ will be denoted by $\mathcal{P}'$. The set $\mathcal{P}'$ is called the set of all \textit{periodic distributions}. The Sobolev space
$H^s_{\rm per}:=H^s_{\rm per}([0,L])$ consists of all  $f \in \mathcal{P}'$ such that
$$
\|f\|^2_{H^s}:= L \sum_{k=-\infty}^{\infty}(1+k^2)^s|\widehat{f}(k)|^2 <\infty
$$
where $\widehat{f}$ is the periodic Fourier transform of $f$. The space $H^s_{\rm per}$ is a  Hilbert space with the inner product denoted by $(\cdot, \cdot)_{H^s}$. When $s=0$, the space $H^s_{\rm per}$ is isometrically isomorphic to the space  $L^2([0,L])$ and will be denoted by $L^2_{\rm per}:=H^0_{\rm per}$ (see, e.g., \cite{Iorio}). The norm and inner product in $L^2_{\rm per}$ will be denoted by $\|\cdot \|_{L^2}$ and $(\cdot, \cdot)_{L^2}$. Explicitly, $(f,g)_{L^2}=\int_0^L f(x) \overline{g(x)}\; dx$ and $\|f\|_{L^2}=\sqrt{(f,f)_{L^2}}$, for all $f, g \in L^2_{\rm per}$.  We need to emphasize here, for $r,s \geq 0$ and given $(f,g) \in H^r_{\rm per} \times H^s_{\rm per}$ the pair $(f,g)$ is composed of two complex functions and so we can write $(f,g)=({\rm Re}\,f, {\rm Re}\,g,{\rm Im}\,f, {\rm Im}\,g).$

For $s\geq0$, we denote
$$
H^s_{\rm per,odd}:=\{ f \in H^s_{\rm per} \; ; \; f \:\; \text{is an odd function}\}.
$$
endowed with the norm and inner product in $H^s_{\rm per}$. In addition,
$$
\mathbb{H}^s_{\rm per}:= H^s_{\rm per} \times H^s_{\rm per} \times H^s_{\rm per} \times H^s_{\rm per} \; \; \text{ and } \; \; \mathbb{H}^s_{\rm per,odd}:=H^s_{\rm per,odd} \times H^s_{\rm per,odd} \times H^s_{\rm per,odd} \times H^s_{\rm per,odd}
$$
equipped with their usual norms and scalar products.


Given $k \in (0, 1)$, the symbols $\sn(\cdot, k), \dn(\cdot, k)$ and $\cn(\cdot, k)$ represent the Jacobi elliptic functions of \textit{snoidal}, \textit{dnoidal}, and \textit{cnoidal} type, respectively. For $\eta \in \left[0, \frac{\pi}{2}\right]$, the notation ${\rm F}(\eta, k)$ and $\E(\eta, k)$  represents the complete elliptic integrals of the first and second kind, respectively, and  we denote by $\K(k)={\rm F}\left(\frac{\pi}{2},k\right)$ and $\E(k)=\E\left(\frac{\pi}{2},k\right)$. For a precise definition and additional properties of elliptic integrals and elliptic Jacobi functions see \cite{byrd}.

\vspace{0.5cm}
\section{Existence of a Smooth Curve of Periodic Waves}\label{SectionExistenceSolutions}

Let $L>0$ be fixed. Our aim in this section is to establish the existence of even and odd $L$-periodic sign-changing solutions $\varphi : \mathbb{R} \to \mathbb{R}$ to the following ordinary differential equation:
\begin{equation}\label{mainEDOquinticsystem}
	-\varphi''+\omega \varphi -(\kappa+\gamma B^3)\varphi^5=0,
\end{equation}
where $\omega>0$.

In this context, concerning the auxiliary equation
\begin{equation}\label{AuxilaryquinticEDO}
	- \phi'' + \omega \phi - \phi^5 = 0,
\end{equation}
for a periodic function $\phi: \mathbb{R} \to \mathbb{R}$ with period $L > 0$, it was shown in \cite[Section 4.3]{NataliCardosoAmaral} (see also \cite{BittencourtLoreno2022}) that for each $k \in (0,1)$, equation \eqref{AuxilaryquinticEDO} admits sign-changing periodic solutions with a cnoidal profile given by
\begin{equation}\label{cnoidalsolution}
	\phi(x) = \frac{a \, \text{cn} \left( \frac{4{\rm K}(k)}{L} x, k \right)}{\sqrt{1 - q \, \text{sn}^2 \left( \frac{4{\rm K}(k)}{L} x, k \right)}}, \quad x \in \mathbb{R}
\end{equation}
where
\begin{equation}\label{avaleuCnoidal}
	a = \frac{2}{L} \left[ \left( 2 - k^2 + 2 \sqrt{k^4 - k^2 + 1}\right) L^2 \, {\rm K}(k)^2 \right]^{\frac{1}{4}}
\end{equation}
and
\begin{equation}\label{qvaleuCnoidal}
	q = -1 + k^2 - \sqrt{k^4 - k^2 + 1}.
\end{equation}
Furthermore, the wave frequency $\omega>0$, which depends on $k \in (0,1)$ and $L>0$, is given by
\begin{eqnarray}\label{valuewCnoidal}
	\omega =  \frac{16{\rm K}(k)^2 \sqrt{k^4 - k^2 +1 }}{L^2}.
\end{eqnarray}

It is important to observe that $\phi$ is an even function. To deal with the problem under different scenarios, we can obtain an odd solution of \eqref{AuxilaryquinticEDO} by considering an $\frac{L}{4}$-translation of $\phi$, namely 
$$
\tilde{\psi}(x) = \phi \left(x+\frac{L}{4}\right), \quad x \in \mathbb{R}.
$$ By employing formulas (120.02), (121.00), and (122.03) from \cite{byrd}, we derive an odd sign-changing solution of \eqref{AuxilaryquinticEDO} given by
\begin{equation}\label{cnoidal-solution-odd}
	\tilde{\psi}(x) = \frac{\tilde{a} \, \text{sn} \left( \frac{4{\rm K}(k)}{L} x, k \right)}{\sqrt{1 - \tilde{q} \, \text{sn}^2 \left( \frac{4{\rm K}(k)}{L} x, k \right)}}, \quad x \in \mathbb{R},
\end{equation}
where 
\begin{equation}\label{tildea-tildeq}
	\tilde{a} = - \frac{a \, \sqrt{1-k^2}}{\sqrt{1-q}} \quad \text{and} \quad \tilde{q} = \frac{k^2 - q}{1-q}.
\end{equation}

Next, let $\phi$ be an $L$-periodic solution of \eqref{AuxilaryquinticEDO}, as defined in \eqref{cnoidalsolution}; we define \begin{equation}\label{transformation-solution}
	\varphi=\theta \phi,
\end{equation}
 where
\begin{equation}\label{theta}
	\theta:=\frac{1}{\sqrt[4]{\kappa+\gamma B^3}},
\end{equation}
and $B>0$ is given in \eqref{Bvalue}. Consequently, $\varphi$ satisfies \eqref{mainEDOquinticsystem} if and only if it is expressed through the transformation $\varphi = \theta \phi$. The same reasoning applies to the odd solution $\psi$ to obtain an odd solution of \eqref{mainEDOquinticsystem}.

We emphasize that, for a fixed $L>0$,  the wave frequency $\omega \in \big( \tfrac{4\pi^2}{L^2}, +\infty \big)$ depends smoothly on $k \in (0,1)$. Based on these considerations regarding the existence of $L$-periodic wave solutions to \eqref{AuxilaryquinticEDO} and the transformation $\varphi = \theta \phi$, we establish the following results:

\begin{proposition}[Existence of Solutions with Cnoidal Profile]\label{cnoidalcurve}
	Let $L>0$ be fixed. Equation \eqref{mainEDOquinticsystem} admits $L$-periodic solutions with a cnoidal profile of the form
	\begin{equation}\label{cnoidal-solution}
		\varphi(x) = \frac{1}{\sqrt[4]{(\kappa + \gamma B^3)}}\frac{ a \, {\rm cn}\left( \frac{4{\rm K}(k)}{L}x,k\right)}{\sqrt{ 1 - q \, {\rm sn}^2 \left( \frac{4 {\rm K}(k)}{L} x, k \right)}},
	\end{equation}
	where the parameters $a \in \left( 2\sqrt{\tfrac{\pi}{L}}, +\infty \right)$, $q \in (-2,-1)$, and the frequency $\omega \in \left( \frac{4\pi^2}{L^2},+\infty\right)$ depend smoothly on $k\in (0,1)$ and $L>0$, as given in \eqref{avaleuCnoidal}, \eqref{qvaleuCnoidal}, and \eqref{valuewCnoidal}, respectively. In addition, the mapping
	$$\omega \in \left( \tfrac{4\pi^2}{L^2},+\infty \right) \longmapsto \varphi_\omega=\varphi \in H^2_{\rm per}$$
	is a smooth curve of periodic solutions for \eqref{mainEDOquinticsystem}.
\end{proposition}

\begin{proposition}[Existence of Solutions with Snoidal Profile]\label{snoidalcurve}
	Let $L>0$ be fixed. Equation \eqref{mainEDOquinticsystem} admits $L$-periodic solutions with a snoidal profile of the form
	\begin{equation}\label{snoidal-solution}
		\psi(x) = \frac{1}{\sqrt[4]{(\kappa + \gamma B^3)}}\frac{ \tilde{a} \, {\rm sn}\left( \frac{4{\rm K}(k)}{L}x,k\right)}{\sqrt{ 1 - \tilde{q} \, {\rm sn}^2 \left( \frac{4 {\rm K}(k)}{L} x, k \right)}},
	\end{equation}
	where the parameters $\tilde{a} \in \left( \frac{-2\sqrt{3}}{3} \sqrt{\frac{\pi}{L}} , 0 \right)$, $\tilde{q} \in \left( \frac{2}{3}, 1 \right)$, and the frequency $\omega \in \left( \frac{4\pi^2}{L^2},+\infty\right)$ depend smoothly on $k\in (0,1)$ and $L>0$, as given in \eqref{tildea-tildeq} and \eqref{valuewCnoidal}, respectively. Moreover, the mapping
	$$\omega \in \left( \tfrac{4\pi^2}{L^2},+\infty \right) \longmapsto \psi_\omega=\psi \in H^2_{\rm per}$$
	is a smooth curve of periodic solutions for \eqref{mainEDOquinticsystem}.
\end{proposition}

\begin{remark}
	It is worth noting that the snoidal solution $\psi$ of \eqref{mainEDOquinticsystem} given in \eqref{snoidal-solution} is an $\frac{L}{4}$-translation of the cnoidal solution $\varphi$ given in \eqref{cnoidal-solution}; specifically, $\psi(x) =  \varphi \left(x + \frac{L}{4} \right)$, for all $x \in \mathbb{R}$.
\end{remark}

\vspace{0.5cm}
\section{Spectral analysis}\label{SectionSpectralAnalysis} 

Let $L>0$ be fixed and consider  the solution $\varphi$ of \eqref{mainEDOquinticsystem} with the cnoidal profile given in \eqref{cnoidal-solution}. In this section, we study  the non-positive spectrum of the linear operator $\mathcal{L}$ defined in \eqref{L-1}. 

Based on the information about the non-positive spectrum of the Hill operators $\mathcal{L}_i$, for $i=1,2,3,4$, given in \eqref{HillOperators}, we determine the number of negative eigenvalues of $\mathcal{L}$, the dimension of its kernel, and identify the elements that constitute ${\rm Ker}(\mathcal{L})$. To achieve this, we utilize key results from Natali \& Cardoso \cite{NataliCardoso2013}, Angulo \& Natali \cite{AnguloNatali2009}, and Moraes \& Loreno \cite{BittencourtLoreno2022}. To ensure clarity, we separate the evaluation of the non-positive spectrum of the operator $\mathcal{L}$ according to the constant $B \in \mathbb{R}$.

A cornerstone of our spectral analysis is to derive a similarity transformation for the operator $\mathcal{L}$, as described in \eqref{similaritymatrixequation} and \eqref{auxiliarylinearizedoperatorrelation}. Therefore, we shall begin this spectral analysis by explicitly providing such a transformation.

\subsection*{Case 1: $B = B_+$.}\label{Spectral analysis- Case i)}

Considering $B=B_+$ as given in \eqref{Bvalue}, we have $\gamma \geq 2\kappa$. However, if $\gamma = 2\kappa$, then $B = 1$, which is another case that will be studied later. Thus, in this case, we only need to consider $\gamma > 2\kappa$. By using Maple software, we are able to obtain an orthogonal matrix $U$ and a matrix $M$ whose entries are given by:

\begin{equation}\label{beta1GeneralB+}
\beta_1 = \frac{1}{4} \frac{3\sqrt{2}\gamma^2 s(\kappa, \gamma) + 7\gamma^4 + 7\gamma^3\sqrt{\gamma^2 - 4\kappa^2} - 24\gamma^2\kappa^2 - 10\gamma\sqrt{\gamma^2 - 4\kappa^2}\kappa^2 + 20\kappa^4}{\kappa^3},
\end{equation}
\begin{equation}\label{beta2GeneralB+}
\beta_2 = \frac{1}{4} \frac{3\sqrt{2}\gamma^2 s(\kappa, \gamma) - \gamma^4 - \gamma^3\sqrt{\gamma^2 - 4\kappa^2} - 2\gamma\sqrt{\gamma^2 - 4\kappa^2}\kappa^2 + 4\kappa^4}{\kappa^3},
\end{equation}
\begin{equation}\label{beta3GeneralB+}
\beta_3 = -\frac{1}{4} \frac{3\sqrt{2}\gamma^2 s(\kappa, \gamma) - 7\gamma^4 - 7\gamma^3\sqrt{\gamma^2 - 4\kappa^2} + 24\gamma^2\kappa^2 + 10\gamma\sqrt{\gamma^2 - 4\kappa^2}\kappa^2 - 20\kappa^4}{\kappa^3}
\end{equation}
and
\begin{equation}\label{beta4GeneralB+}
\beta_4 = -\frac{1}{4} \frac{3\sqrt{2}\gamma^2 s(\kappa, \gamma) + \gamma^4 + \gamma^3\sqrt{\gamma^2 - 4\kappa^2} + 2\gamma\sqrt{\gamma^2 - 4\kappa^2}\kappa^2 - 4\kappa^4}{\kappa^3},
\end{equation}
where
\begin{equation}\label{skappagamma}
s(\kappa, \gamma) := \sqrt{\gamma^4 + \gamma^3\sqrt{\gamma^2 - 4\kappa^2} - 4\gamma^2\kappa^2 - 2\gamma\sqrt{\gamma^2 - 4\kappa^2}\kappa^2 + 2\kappa^4}.
\end{equation}

The entries $\beta_i$, for each $i \in \{1,2,3,4$\}, of the matrix $M$ are obtained by calculating the eigenvalues of the matrix $S$ given in \eqref{Smatrix}. Consequently, the desired matrix $U$ has as its columns the eigenvectors associated with $\beta_i$, normalized to form an orthonormal set of vectors. So, $U$ is an orthogonal matrix, satisfying $U^{-1} = U^T = U^*$, $\det(U) = 1$, \eqref{similaritymatrixequation} and \eqref{auxiliarylinearizedoperatorrelation}.

Again, since $\gamma > 2\kappa$, we can use Maple to perform algebraic manipulations and conclude that the following chain of inequalities holds:
\begin{equation}\label{betainequalities}
    \beta_4 < \beta_2 < \beta_3 < \beta_1 \quad \text{and} \quad \beta_4 < 0.
\end{equation}

\subsection*{Case 2: $B = B_-$.}

Now, considering $B = B_-$ given in \eqref{Bvalue}, we have the same scenario $\gamma > 2\kappa$ as discussed in \textbf{Case 1}. Then, with the help of Maple software, we are able to obtain the entries of the matrix $M$ in \eqref{Mmatrix} which are given by:
\begin{equation}\label{beta1GeneralB-}
\beta_1=\frac{1}{4}\,{\frac {3\,\sqrt {2}{\gamma}^{2}\; s(\kappa,\gamma)+7\,{\gamma}^{4}-7\,
{\gamma}^{3}\sqrt {{\gamma}^{2}-4\,{k}^{2}}-24\,{\gamma}^{2}{k}^{2}+10
\,\gamma\,\sqrt {{\gamma}^{2}-4\,{k}^{2}}{k}^{2}+20\,{k}^{4}}{{k}^{3}}
},
\end{equation}
\begin{equation}\label{beta2GeneralB-}
\beta_2=\frac{1}{4}\,{\frac {3\,\sqrt {2}{\gamma}^{2}\; s(\kappa,\gamma)-{\gamma}^{4}+{
\gamma}^{3}\sqrt {{\gamma}^{2}-4\,{k}^{2}}+2\,\gamma\,\sqrt {{\gamma}^
{2}-4\,{k}^{2}}{k}^{2}+4\,{k}^{4}}{{k}^{3}}},
\end{equation}
\begin{equation}\label{beta3GeneralB-}
\beta_3=-\frac{1}{4}\,{\frac {3\,\sqrt {2}{\gamma}^{2}\; s(\kappa,\gamma)-7\,{\gamma}^{4}+7\,
{\gamma}^{3}\sqrt {{\gamma}^{2}-4\,{k}^{2}}+24\,{\gamma}^{2}{k}^{2}-10
\,\gamma\,\sqrt {{\gamma}^{2}-4\,{k}^{2}}{k}^{2}-20\,{k}^{4}}{{k}^{3}}
}
\end{equation}
and 
\begin{equation}\label{beta4GeneralB-}
\beta_4=-\frac{1}{4}\,{\frac {3\,\sqrt {2}{\gamma}^{2}\; s(\kappa,\gamma)+{\gamma}^{4}-{
\gamma}^{3}\sqrt {{\gamma}^{2}-4\,{k}^{2}}-2\,\gamma\,\sqrt {{\gamma}^
{2}-4\,{k}^{2}}{k}^{2}-4\,{k}^{4}}{{k}^{3}}}.
\end{equation}
Here, $s(\kappa, \gamma)$ is given in \eqref{skappagamma}. Despite the similarity, it is straightforward to see that the values $\beta_i$ in \eqref{beta1GeneralB+}-\eqref{beta4GeneralB+} are different from those shown in \eqref{beta1GeneralB-}-\eqref{beta4GeneralB-}. Again, by using Maple, we show that the eigenvalues in \eqref{beta1GeneralB-}-\eqref{beta4GeneralB-} also satisfy the inequality in \eqref{betainequalities}. In addition, the desired matrix $U$ has as its columns the eigenvectors associated with $\beta_i$, $i \in \{1,2,3,4\}$, normalized so as to form an orthonormal set. Consequently, $U$ is an orthogonal matrix satisfying
$
U^{-1} = U^{T} = U^{*}, \det(U) = 1,
$
and relations \eqref{similaritymatrixequation} and \eqref{auxiliarylinearizedoperatorrelation} hold.

\subsection*{Case 3: $B=1$}\label{Spectral analysis- Case iii)}

One of the solutions of equation \eqref{Constantsassumptions} is $B=1$. In this case, we initially do not have any condition on $\gamma \geq 0$ or $\kappa > 0$. Using Maple, we calculate the eigenvalues of the matrix $S$ given in \eqref{Smatrix} to conclude that the matrix $M$ is given by
\begin{equation*}
M= \begin{bmatrix}
5\kappa+5\gamma & 0 & 0 & 0 \\ 
0 & 5\kappa-\gamma & 0 & 0 \\ 
0 & 0 & \kappa+\gamma & 0 \\ 
0 & 0 & 0 & \kappa-5\gamma
\end{bmatrix}.
\end{equation*}
The entries of the matrix $M$ are:
\begin{equation}
    \beta_1 = 5\kappa + 5\gamma, \quad \beta_2 = \kappa+\gamma, \quad \beta_3 = 5\kappa-\gamma, \quad \text{ and } \quad \beta_4 = \kappa - 5 \gamma.
\end{equation}
As in the previous cases, we obtain an orthogonal matrix $U$ such that $U^{-1}=U^T=U^*$ and $\det(U)=1$ using the eigenvectors associated to $\beta_i$. To derive a chain of inequalities for $\beta_i$ as in the preceding cases, one must consider different scenarios for $\gamma$. More precisely, we study $\gamma$ into four initial scenarios: $\gamma = 0$, $\gamma \in (0,2\kappa)$, $\gamma = 2\kappa$, and $\gamma \in (2\kappa, +\infty)$. Thus, we analyze the subcases of $B=1$ to investigate the non-positive spectrum of $\mathcal{L}$. 

\subsubsection*{Case 3.1: $\gamma = 0$.} 
In this case, we have $\beta_1 = 5\kappa$, $\beta_2 = \kappa$, $\beta_3 = 5\kappa$, and $\beta_4 = \kappa$. Thus, $\beta_1 = \beta_3$, $\beta_2 = \beta_4$ and, consequently, we have that $\beta_4 = \beta_2 < \beta_3 = \beta_1$.

\subsubsection*{Case 3.2: $\gamma \in (0,2\kappa)$.} 
Here, considering $\gamma < 2 \kappa$, we can conclude that $\beta_4 < \beta_2 < \beta_3 < \beta_1$.

\subsubsection*{Case 3.3: $\gamma = 2\kappa$.} 
With $\gamma = 2\kappa$, we have $\beta_2 = \beta_3$, obtaining $\beta_4 < \beta_2 = \beta_3 < \beta_1$.

\subsubsection*{Case 3.4: $\gamma \in (2\kappa, +\infty)$.} 
In this case, there is a swap between $\beta_2$ and $\beta_3$, namely, $\beta_4 < \beta_3 < \beta_2 < \beta_1$.

Therefore, if $B = 1$, we have different chains of inequalities for $\beta_i$ depending on $\gamma$. This is fundamental for our spectral analysis. Finally, we consider the case $B = -1$.

\subsection*{Case 4: $B = -1$.} 
Now, we consider $\gamma \in [0,\kappa)$ due to the constant $\theta$ given in \eqref{theta}. Using Maple, we conclude that the matrix $M$ is given by
\begin{equation*}
M= \begin{bmatrix}
5\kappa-5\gamma & 0 & 0 & 0 \\ 
0 & 5\kappa+\gamma & 0 & 0 \\ 
0 & 0 & \kappa-\gamma & 0 \\ 
0 & 0 & 0 & \kappa+5\gamma
\end{bmatrix}.
\end{equation*}
The entries of the matrix $M$ are:
\begin{equation}
    \beta_1 = 5\kappa - 5\gamma, \quad \beta_2 = \kappa-\gamma, \quad \beta_3 = 5\kappa+\gamma, \quad \text{ and } \quad \beta_4 = \kappa + 5 \gamma.
\end{equation}
As before, we obtain an orthogonal matrix $U$ using the normalized eigenvectors associated to $\beta_i$. In this case, we have two partitions of $\gamma$: $\gamma = 0$ and $\gamma \in (0,\kappa)$. If $\gamma = 0$, then $\beta_1 = \beta_3$ and $\beta_2 = \beta_4$, concluding that $\beta_4 = \beta_2 < \beta_3 = \beta_1$. On the other hand, for $\gamma \in (0,\kappa)$, we conclude that $\beta_2 < \beta_4 < \beta_1 < \beta_3$ for $\gamma \in \left(0, \frac{2}{5} \kappa\right)$ and $\beta_2 < \beta_1 < \beta_4 < \beta_3$ for $\gamma \in \left(\frac{2}{5}\kappa, \kappa\right)$.

From the above, we have obtained the linearization of the operator $\mathcal{L}$, the entries $\beta_i$ of the matrix $M$, and the chain of inequalities for $\beta_i$ for each case of $B$ and possible subcases of $\gamma$. We are now prepared to analyze the non-positive spectrum of the operator $\mathcal{L}$ for each case evaluated previously.

\subsection{Spectral analysis for $\mathcal{L}$ in $\mathbb{H}^2_{\rm per}$ -- solution with cnoidal profile}\label{spectral-analysis-cnoidal}

After explicitly providing a similarity transformation of $\mathcal{L}$ for the different cases of $B$ presented in \eqref{Bvalue}, we are able to evaluate the non-positive spectrum of the operator $\mathcal{L}$. More precisely, we perform this analysis considering $\varphi$ as the cnoidal solution of \eqref{mainEDOquinticsystem} given in \eqref{cnoidal-solution}. Using the transformations established thus far, we aim to obtain the characterization of the non-positive spectrum of each Hill operator $\mathcal{L}_i = -\partial_x^2 + \omega - \beta_i \varphi^4$, for $i=1,2,3,4$.

Initially, it is crucial to address the analysis of the operators $\tilde{\mathcal{L}}_1$ and $\tilde{\mathcal{L}}_2$ given by
\begin{equation*}
    \tilde{\mathcal{L}}_1 = -\partial_x^2 + \omega \phi - 5\phi^4 \quad \text{and} \quad \tilde{\mathcal{L}}_2 = -\partial_x^2 + \omega \phi - \phi^4,
\end{equation*}
where $\phi$ is the solution of \eqref{AuxilaryquinticEDO} with the cnoidal profile given in \eqref{cnoidalsolution}. More precisely, following the arguments presented in \cite{BittencourtLoreno2022}, we have that 
\begin{equation}\label{L1-normal}
    {\rm n}(\tilde{\mathcal{L}}_1) = 2, \quad {\rm z}(\tilde{\mathcal{L}}_1) = 1, \quad \text{and} \quad {\rm Ker}(\tilde{\mathcal{L}}_1) = [\phi']
\end{equation}
and
\begin{equation}\label{L2-normal}
    {\rm n}(\tilde{\mathcal{L}}_2) = 1, \quad {\rm z}(\tilde{\mathcal{L}}_2) = 1, \quad \text{and} \quad {\rm Ker}(\tilde{\mathcal{L}}_2) = [\phi].
\end{equation}

By using \eqref{transformation-solution}, \eqref{L1-normal}, and \eqref{L2-normal}, and performing algebraic simplifications with the aid of Maple software, we conclude that
\begin{equation*}
    \mathcal{L}_1 = -\partial_x^2 + \omega - \beta_1 \varphi^4 = \tilde{\mathcal{L}}_1 \quad \text{ and } \quad \mathcal{L}_2 = -\partial_x^2 + \omega - \beta_2 \varphi^4 = \tilde{\mathcal{L}}_2.
\end{equation*}

Therefore, it follows that 
\begin{equation}\label{nL1}
    {\rm n}({\mathcal{L}}_1) = 2, \quad {\rm z}({\mathcal{L}}_1) = 1, \quad \text{and} \quad {\rm Ker}({\mathcal{L}}_1) = [\varphi']
\end{equation}
and
\begin{equation}\label{nL2}
    {\rm n}({\mathcal{L}}_2) = 1, \quad {\rm z}({\mathcal{L}}_2) = 1, \quad \text{and} \quad {\rm Ker}({\mathcal{L}}_2) = [\varphi].
\end{equation}

Using the known results for the Hill operators $\mathcal{L}_1$ and $\mathcal{L}_2$, we are able to apply the Comparison Theorem in the Periodic Case (see \cite[Theorem 2.2.2]{Eastham} or also \cite[Appendix]{Pastor}) to obtain the behavior of the non-positive spectrum of $\mathcal{L}_3$ and $\mathcal{L}_4$ in each case of $B$ and, consequently, the non-positive spectrum of $\mathcal{L}$. For this, we need to separate our analysis in the four cases of $B$ again.

\subsection*{Case 1: $B = B_+$} Considering $B = B_+$ given in \eqref{Bvalue}, we have that $\beta_4 < \beta_2 < \beta_3 < \beta_1$ and $\beta_4 < 0$. Then, we can conclude that $\mathcal{L}_1 < \mathcal{L}_3 < \mathcal{L}_2 < \mathcal{L}_4$ where $\mathcal{L}_i < \mathcal{L}_j$ means that $(\mathcal{L}_i u,u)_{L^2} < (\mathcal{L}_j u,u)_{L^2}$ for all $u \in H^2_{\rm per}$, $u \neq 0$ and $i,j \in \{1,2,3,4\}$, $i \neq j$. Thus, using the Comparison Theorem in the Periodic Case with \eqref{nL1} and \eqref{nL2}, we have that ${\rm n}(\mathcal{L}_3) = 2$ and ${\rm z}(\mathcal{L}_3) = 0$. In addition, since $\beta_4 < 0$, we obtain that ${\rm n}(\mathcal{L}_4) = {\rm z}(\mathcal{L}_4) = 0$. Therefore, considering $B = B_+$, we conclude that 
\begin{equation}\label{spectralanalysis-B+-B-}
    {\rm n}(\mathcal{L}) = 5, \quad {\rm z}(\mathcal{L}) = 2 \quad \text{ and } \quad {\rm Ker}(\mathcal{L}) = [(\varphi', B\varphi',0,0), (0,0,\varphi,B\varphi)].
\end{equation}

\subsection*{Case 2: $B = B_-$} Here, since we have the same chain of inequalities of $\beta_i$, then we can perform the same analysis to conclude that the nonnegative spectrum of $\mathcal{L}$ also satisfies \eqref{spectralanalysis-B+-B-}.

\subsection*{Case 3: $B = 1$} This is the most complex scenario, requiring the analysis to be divided into four distinct partitions of $\gamma$. Although we can establish a chain of inequalities for each partition, these results alone do not suffice for our objectives. We shall examine, in each subcase, the underlying reasons for this insufficiency.

\subsubsection*{Case 3.1: $\gamma = 0$} In this case, since $\beta_4 = \beta_2 < \beta_3 = \beta_1$, we conclude that $\mathcal{L}_1 = \mathcal{L}_3 < \mathcal{L}_2 = \mathcal{L}_4$. Then, using \eqref{nL1} and \eqref{nL2}, we conclude that
\begin{equation*}
    {\rm n}(\mathcal{L}) = 6, \quad {\rm z}(\mathcal{L}) = 4 \quad \text{ and } \quad {\rm Ker}(\mathcal{L}) = [(\varphi',0,0,0), (0,\varphi',0,0), (0,0,\varphi,0), (0,0,0,\varphi)].
\end{equation*}

\subsubsection*{Case 3.2: $\gamma \in (0,2\kappa)$} Here, we have that $\beta_4 < \beta_2 < \beta_3 < \beta_1$. Since $\beta_3 \in (\beta_2, \beta_1)$, we use \eqref{nL1}, \eqref{nL2} and the Comparison Theorem to obtain that ${\rm n}(\mathcal{L}_3) = 2$ and ${\rm z}(\mathcal{L}_3) = 0$. However, we are able to conclude some result of $\mathcal{L}_4$ only if $\beta_4 < 0$ and this happens when $\gamma > \frac{\kappa}{5}$. More precisely, in this scenario, we have that $\mathcal{L}_4 > 0$, implying that ${\rm n}(\mathcal{L}_4) = {\rm z}(\mathcal{L}_4) = 0$. Therefore, considering $B = 1$ and $\gamma \in \left(\frac{\kappa}{5}, 2\kappa\right)$, we conclude that 
\begin{equation*}
    {\rm n}(\mathcal{L}) = 5, \quad {\rm z}(\mathcal{L}) = 2 \quad \text{ and } \quad {\rm Ker}(\mathcal{L}) = [(\varphi', B\varphi',0,0), (0,0,\varphi,B\varphi)].
\end{equation*}

\subsubsection*{Case 3.3: $\gamma = 2\kappa$} In this case, we have that $\beta_2 = \beta_3$ and, in addition, $\beta_4 < \beta_3 = \beta_2 < \beta_1$. This implies that $\mathcal{L}_1 < \mathcal{L}_2 = \mathcal{L}_3 < \mathcal{L}_4$. Since $\beta_2 = \beta_3$, the properties of $\mathcal{L}_2$ are extended to $\mathcal{L}_3$. On the other hand, as $\beta_4 = \kappa - 5 \gamma$, we have that $\beta_4 = - 9 \gamma < 0$ implying that $\mathcal{L}_4 > 0$. Under the above, we conclude that
\begin{equation}\label{spectralanalysis-B1-gamma2kappa}
    {\rm n}(\mathcal{L}) = 4, \quad {\rm z}(\mathcal{L}) = 3 \quad \text{ and } \quad {\rm Ker}(\mathcal{L}) = [(\varphi', B\varphi',0,0), (0,0,\varphi,B\varphi), (-B\varphi, \varphi,0,0)].
\end{equation}

\subsubsection*{Case 3.4: $\gamma \in (2\kappa, +\infty)$} In this last case, we obtained that $\beta_4 < \beta_3 < \beta_2 < \beta_1$. Since $\beta_3$ is not between $\beta_1$ and $\beta_2$, the Comparison Theorem cannot be conveniently applied here. On the other hand, since $\beta_3 = 5\kappa - \gamma$, we have that $\beta_3 < 0$ only if $\gamma > 5\kappa$. In this case, we have that $\mathcal{L}_3 > 0$ and $\mathcal{L}_4 > 0$. Then, considering $\gamma \in (5\kappa, +\infty)  \subset (2\kappa, +\infty)$, we conclude that 
\begin{equation*}
    {\rm n}(\mathcal{L}) = 3, \quad {\rm z}(\mathcal{L}) = 2 \quad \text{ and } \quad {\rm Ker}(\mathcal{L}) = [(\varphi', B\varphi',0,0), (0,0,\varphi,B\varphi)].
\end{equation*}

\subsection*{Case 4: $B = -1$} In this scenario, we distinguish between two distinct cases for $\gamma$: $\gamma = 0$ and $\gamma \in (0,\kappa)$. On one hand, for $\gamma = 0$, it follows that $\beta_1 = \beta_3$ and $\beta_2 = \beta_4$. Consequently, analogous to Case 3.1 ($B=1$), we conclude that:
\begin{equation}\label{spectralanalysis-B1-gamma0}
    {\rm n}(\mathcal{L}) = 6, \quad {\rm z}(\mathcal{L}) = 4 \quad \text{ and } \quad {\rm Ker}(\mathcal{L}) = [(\varphi',0,0,0), (0,\varphi',0,0), (0,0,\varphi,0), (0,0,0,\varphi)].
\end{equation}
On the other hand, for $\gamma \in (0,\kappa)$, there exist eigenvalues $\beta_i$ such that $\beta_i > \beta_1$ implying that\footnote{  { For instance: $\mathcal{L}_3<\mathcal{L}_1$ if $0<\gamma<  \frac{2}{5}\kappa$ and $\mathcal{L}_4<\mathcal{L}_1$ if $\frac{2}{5}\kappa<\gamma< \kappa$.}} $\mathcal{L}_i < \mathcal{L}_1$. Under these conditions, the Comparison Theorem together with the spectral properties in \eqref{nL1} and \eqref{nL2} are not sufficient to characterize the non-positive spectrum of $\mathcal{L}_3$ and $\mathcal{L}_4$. Therefore, this particular case remains open in the literature.

\subsection{Spectral analysis for $\mathcal{L}$ in $H^2_{\rm per,odd}$ -- solution with snoidal profile}\label{spectral-analysis-snoidal}

Our aim now is to evaluate the non-positive spectrum of the operator $\mathcal{L}$ given in \eqref{L-1}, restricted to the subspace of odd functions $H^2_{\rm per,odd}$. To address this problem, we begin by evaluating the non-positive spectrum of $\mathcal{L} : \mathbb{H}^2_{\rm per} \rightarrow \mathbb{H}^2_{\rm per}$ defined by $\mathcal{L} = \left(-\partial_x^2 + \omega\right) {\rm Id}_{4 \times 4} - \psi^4 S$, where $S$ is given in \eqref{Smatrix} and $\psi$ is the snoidal solution of \eqref{mainEDOquinticsystem} obtained in \eqref{snoidal-solution}. In this case, since $\psi(x) = \varphi\left( x + \frac{L}{4} \right)$, for all $x \in \mathbb{R}$, the properties established in \eqref{nL1} and \eqref{nL2} extend to $\mathcal{L}_1 = -\partial_x^2 + \omega - \beta_i \psi^4$ for $i=1,2$. Consequently, we conclude that
\begin{equation}\label{nL1-psi}
    {\rm n}(\mathcal{L}_1) = 2, \quad {\rm z}(\mathcal{L}_1) = 1 \quad \text{ and } \quad {\rm Ker}(\mathcal{L}_1) = [\psi']
\end{equation}
and
\begin{equation}\label{nL2-psi}
    {\rm n}(\mathcal{L}_2) = 1, \quad {\rm z}(\mathcal{L}_2) = 1 \quad \text{ and } \quad {\rm Ker}(\mathcal{L}_2) = [\psi].
\end{equation}

However, we seek to evaluate the behavior of the non-positive spectrum of $\mathcal{L}$ in $\mathbb{H}^2_{\rm per,odd}$. By considering the restriction of $\mathcal{L}$ to the subspace of odd functions, we obtain the operator $\mathcal{L}_{\rm odd}: \mathbb{H}^2_{\rm per,odd} \rightarrow \mathbb{L}^2_{\rm per,odd}$, where $\mathcal{L}_{\rm odd} = U \tilde{\mathcal{L}}_{\rm odd} U^{-1}$ with
\begin{equation*}
 {\tilde{\mathcal{L}}_{\rm odd}}= \begin{bmatrix}
\mathcal{L}_{\rm 1,odd} & 0 & 0 & 0 \\
0 & \mathcal{L}_{\rm 3,odd} & 0 & 0 \\
0 & 0 & \mathcal{L}_{\rm 2,odd} & 0 \\
0 & 0 & 0 & \mathcal{L}_{\rm 4,odd}
\end{bmatrix}
\end{equation*}
and, for each $i \in \{1,2,3,4\}$, the operators $\mathcal{L}_{i,\rm odd}:H^2_{\rm per, odd} \subset L^2_{\rm per, odd} \longrightarrow L^2_{\rm  per, odd}$ are given by
\begin{equation*}
\mathcal{L}_{i,\rm odd} = -\partial_x^2 + \omega - \beta_i \psi^4,
\end{equation*}
where $\beta_i \in \mathbb{R}$ for $i \in \{1,2,3,4\}$ depends on the choice of $B$ in \eqref{Bvalue}.

Starting our analysis on the subspace of odd functions, and using \eqref{nL1-psi}, \eqref{nL2-psi}, and the fact that $\psi$ is an odd function, it follows that
\begin{equation*}
    {\rm z}(\mathcal{L}_{1,\rm odd}) = 0 \text{ with } {\rm Ker}(\mathcal{L}_{1,\rm odd}) = \{0\} \quad \text{ and } \quad {\rm z}(\mathcal{L}_{2,\rm odd}) = 1 \text{ with } {\rm Ker}(\mathcal{L}_{2,\rm odd}) = [\psi].
\end{equation*}
In addition, since $\mathcal{L}_{1,\rm odd}$ and $\mathcal{L}_{2,\rm odd}$ are Hill operators of the form $-\partial_x^2 + V(x)$, where $V$ is an even functional, we can apply the Floquet Theory in the periodic context as presented by Neves \cite{Neves} to conclude that the eigenfunction associated with the first eigenvalue of each operator is even. Also,  the eigenfunction associated to the second negative eigenvalue of the operator $\mathcal{L}_1$ should be a odd function.  Thus,  ${\rm n}(\mathcal{L}_{1,\rm odd}) \leq 1$ and  ${\rm n}(\mathcal{L}_{2,\rm odd}) = 0$. On the other hand,  since in any case we have $\beta_2<\beta_1$ by the Comparison Theorem and the fact that $\mathcal{L}_{1,\rm odd} < \mathcal{L}_{2,\rm odd}$, we have ${\rm n}(\mathcal{L}_{1,\rm odd}) = 1$. Summarizing the above, we conclude that
\begin{equation}\label{nL1-odd}
    {\rm n}(\mathcal{L}_{1,\rm odd}) = 1, \quad {\rm z}(\mathcal{L}_{1,\rm odd}) = 0 \quad \text{ and } \quad {\rm Ker}(\mathcal{L}_{1,\rm odd}) = \{0\}
\end{equation}
and
\begin{equation}\label{nL2-odd}
    {\rm n}(\mathcal{L}_{2,\rm odd}) = 0, \quad {\rm z}(\mathcal{L}_{2,\rm odd}) = 1 \quad \text{ and } \quad {\rm Ker}(\mathcal{L}_{2,\rm odd}) = [\psi].
\end{equation}
Therefore, we can observe that the number of negative eigenvalues of $\mathcal{L}_{1,\rm odd}$ and $\mathcal{L}_{2,\rm odd}$ decreases in comparison with $\mathcal{L}_1$ and $\mathcal{L}_2$. This fact allows us to use the Comparison Theorem in the periodic context \cite[Theorem 2.2.2]{Eastham} to establish the properties of the non-positive spectrum of $\mathcal{L}_{3,\rm odd}$ and $\mathcal{L}_{4,\rm odd}$ in each case of $B$ present in \eqref{Bvalue}.

\subsection*{Case 1: $B = B_+$} 
Considering $B = B_+$ as given in \eqref{Bvalue}, we have $\beta_4 < \beta_2 < \beta_3 < \beta_1$, which implies $\mathcal{L}_1 < \mathcal{L}_3 < \mathcal{L}_2 < \mathcal{L}_4$. By applying the Comparison Theorem with \eqref{nL1-odd} and \eqref{nL2-odd}, we find that ${\rm n}(\mathcal{L}_{3,\rm odd}) = 1$ and ${\rm z}(\mathcal{L}_{3,\rm odd}) = 0$. Additionally, we obtain ${\rm n}(\mathcal{L}_{4,\rm odd}) = {\rm z}(\mathcal{L}_{4,\rm odd}) = 0$. Therefore, for $B = B_+$, we conclude that 
\begin{equation}\label{spectralanalysis-B+-B-odd}
    {\rm n}(\mathcal{L}_{\rm odd}) = 2, \quad {\rm z}(\mathcal{L}_{\rm odd}) = 1 \quad \text{ and } \quad {\rm Ker}(\mathcal{L}_{\rm odd}) = [(0,0,\psi,B\psi)].
\end{equation}

\subsection*{Case 2: $B = B_-$} 
Since we have the same chain of inequalities for $\beta_i$ as in the case $B = B_+$, it follows that the non-positive spectrum of $\mathcal{L}$ also satisfies \eqref{spectralanalysis-B+-B-odd}.

\subsection*{Case 3: $B = 1$} 
In this case, the analysis must be divided into four distinct partitions of $\gamma$. 

\subsubsection*{Case 3.1: $\gamma = 0$} 
Here, since $\beta_4 = \beta_2 < \beta_3 = \beta_1$, we obtain $\mathcal{L}_1 = \mathcal{L}_3 < \mathcal{L}_2 = \mathcal{L}_4$. From \eqref{nL1-odd} and \eqref{nL2-odd}, we conclude that
\begin{equation*}
    {\rm n}(\mathcal{L}_{\rm odd}) = 2, \quad {\rm z}(\mathcal{L}_{\rm odd}) = 2 \quad \text{ and } \quad {\rm Ker}(\mathcal{L}_{\rm odd}) = [(0,0,\psi,0), (0,0,0,\psi)].
\end{equation*}

\subsubsection*{Case 3.2: $\gamma \in (0,2\kappa)$} 
In this subcase, we have $\beta_4 < \beta_2 < \beta_3 < \beta_1$. Using \eqref{nL1-odd}, \eqref{nL2-odd}, and the Comparison Theorem, we obtain ${\rm n}(\mathcal{L}_{3,\rm odd}) = 1$ and ${\rm z}(\mathcal{L}_{3,\rm odd}) = 0$. Since the first eigenvalue of ${\mathcal{L}_{2,odd}}$ is zero, the Comparison Theorem guarantees that ${\rm n}(\mathcal{L}_{4,\rm odd}) = 0$ and ${\rm z}(\mathcal{L}_{4,\rm odd}) = 0$. Thus, comparing this with Case 3.2 in $H^2_{\rm per}$, we conclude that 
\begin{equation*}
    {\rm n}(\mathcal{L}_{\rm odd}) = 2, \quad {\rm z}(\mathcal{L}_{\rm odd}) = 1 \quad \text{ and } \quad {\rm Ker}(\mathcal{L}_{\rm odd}) = [(0,0,\psi,B\psi)]
\end{equation*}
over the whole interval $(0,2\kappa)$.

\subsubsection*{Case 3.3: $\gamma = 2\kappa$} 
In this case, we have $\beta_2 = \beta_3$ and $\beta_4 < \beta_3 = \beta_2 < \beta_1$, which implies $\mathcal{L}_1 < \mathcal{L}_2 = \mathcal{L}_3 < \mathcal{L}_4$. Since $\beta_2 = \beta_3$, the properties of $\mathcal{L}_{2,\rm odd}$ extend to $\mathcal{L}_{3,\rm odd}$. Again, using \eqref{nL2-odd} and the Comparison Theorem, we obtain ${\rm n}(\mathcal{L}_{4,\rm odd}) = {\rm z}(\mathcal{L}_{4,\rm odd}) = 0$. Consequently, we conclude that
\begin{equation*}
    {\rm n}(\mathcal{L}_{\rm odd}) = 1, \quad {\rm z}(\mathcal{L}_{\rm odd}) = 2 \quad \text{ and } \quad {\rm Ker}(\mathcal{L}_{\rm odd}) = [(0,0,\psi,B\psi), (-B\psi, \psi,0,0)].
\end{equation*}

\subsubsection*{Case 3.4: $\gamma \in (2\kappa, +\infty)$} 
In this final case, we have $\beta_4 < \beta_3 < \beta_2 < \beta_1$. From \eqref{nL2-odd} and the Comparison Theorem, it follows that ${\rm n}(\mathcal{L}_{3,\rm odd}) = {\rm z}(\mathcal{L}_{3,\rm odd}) = {\rm n}(\mathcal{L}_{4,\rm odd}) = {\rm z}(\mathcal{L}_{4,\rm odd}) = 0$. Thus,
\begin{equation}\label{spectralanalysis-B1-gamma-large}
    {\rm n}(\mathcal{L}_{\rm odd}) = 1, \quad {\rm z}(\mathcal{L}_{\rm odd}) = 1 \quad \text{ and } \quad {\rm Ker}(\mathcal{L}_{\rm odd}) = [(0,0,\psi,B\psi)].
\end{equation}
Thus, comparing this with Case 3.4 in $H^2_{\rm per}$, we also obtain this result over the whole interval $(2\kappa, +\infty)$. 

\subsection*{Case 4: $B = -1$} 
In this scenario, we consider two distinct cases for $\gamma$: $\gamma = 0$ and $\gamma \in (0,\kappa)$. For $\gamma = 0$, it follows that $\beta_1 = \beta_3$ and $\beta_2 = \beta_4$. Consequently, analogous to Case 3.1 ($B=1$), we conclude that:
\begin{equation}\label{spectralanalysis-B-minus1-gamma0}
    {\rm n}(\mathcal{L}_{\rm odd}) = 2, \quad {\rm z}(\mathcal{L}_{\rm odd}) = 2 \quad \text{ and } \quad {\rm Ker}(\mathcal{L}_{\rm odd}) = [(0,0,\psi,0), (0,0,0,\psi)].
\end{equation}
Conversely, for $\gamma \in (0,\kappa)$, the difficulty encountered in $H^2_{\rm per}$ also arises here, as there exist eigenvalues $\beta_i$ such that $\beta_i > \beta_1$, implying $\mathcal{L}_i < \mathcal{L}_1$. In this situation, the Comparison Theorem and other established spectral results do not provide sufficient information to characterize the non-positive spectrum of $\mathcal{L}_i$. Therefore, this particular case remains an open problem in the literature.

\vspace{0.5cm}
\section{Spectral stability}\label{Spectral stability Section}
In this section, we establish spectral stability results for the periodic multiple solution $\Phi = (\varphi, B\varphi,0,0)$, where $\varphi \in H^1_{\rm per}$ is assumed to have a cnoidal profile. To this end, it is necessary to determine the entries of the matrix $V \in \mathbf{M}_{n \times n}(\mathbb{R})$ defined in \eqref{Ventries}, where $n ={\rm z}(\mathcal{L})$. Specifically, we derive a simplified expression for the entries of $V$ by employing the transformation \eqref{adjointrelatioapplication} and the fact that $\Theta_i \in {\rm Ker}(\mathcal{L})$ for each $i \in \{1, 2, \dots, n\}$, namely
\begin{equation}\label{VentriesExplicity}
V_{ij} = \left( \mathcal{L}^{-1} J \Theta_i, J \Theta_j \right)_{L^2 \times L^2} = \left( \tilde{\mathcal{L}}^{-1} U^{-1} J \Theta_i, U^{-1} J \Theta_j \right)_{L^2 \times L^2}, \quad i,j \in \{1, 2, \dots, n\}.
\end{equation}

Our stability analysis further relies on the spectral properties of the operator $\mathcal{L}$ in \eqref{L-1}, as characterized in the previous section. Before analyzing the several settings of the matrix $V$ induced by the choice of $B$ in \eqref{Bvalue}, we provide some fundamental results concerning the Hill operators $\mathcal{L}_i$ for $i=1,2,3,4$. These preliminary are essential to the reader’s understanding and help streamline the subsequent computation of  ${\rm n}(V)$.

\begin{lemma}\label{positivity}
Let $H=(H,(\cdot,\cdot)_H)$ be a Hilbert space and  $\mathcal{A}:{\rm D}(\mathcal{A}) \subset H \to H$ be a densely defined and self-adjoint operator. If the spectrum $\sigma(\mathcal{A})$ of $\mathcal{A}$ is constituted only by an infinite discrete set of eigenvalues and satisfying $\sigma(\mathcal{A})\subset[0,+\infty)$, then there exists $\delta > 0$ such that 
\begin{equation}
	(\mathcal{A} v, v)_{H} \geq \delta \|v\|_{H}^2,\: \text{for all}\: v \in {\rm D}(\mathcal{A}) \cap {\rm Ker}(\mathcal{A})^\perp.
\end{equation}
\end{lemma}
\begin{proof}
    The proof of this result can be found in \cite[Remark 4.1]{NataliBittencourt2024}.
\end{proof}

\begin{lemma}\label{dif-norma}
Let $L>0$ and $\omega \in \left( \frac{4\pi^2}{L^2}, +\infty \right)$ be fixed. {If  $\varphi=\varphi_\omega$ is the periodic solution with cnoidal profile given in Proposition \ref{cnoidalcurve}, then   $\frac{d}{d\omega} \|\varphi\|_{L^2}^2 > 0$.}
\end{lemma}

\begin{proof}
    Firstly, we note that $\omega$ depends smoothly on $k \in (0,1)$. By the chain rule, we have
    \begin{equation*}
        \frac{d}{d\omega} \|\varphi\|_{L^2}^2 = \frac{d}{dk} \|\varphi\|_{L^2}^2 \, \frac{dk}{d\omega} = \frac{d}{dk} \|\varphi\|_{L^2}^2 \, \left( \frac{d\omega}{dk}\right)^{-1}.
    \end{equation*}
    Our goal is to show the positivity of both $\frac{d}{dk} \|\varphi\|_{L^2}^2$ and $\frac{d\omega}{dk}$. First, using the expression for $\varphi$ given in \eqref{cnoidal-solution} and a change of variables, we obtain
    \begin{equation*}
        \|\varphi\|_{L^2}^2 = \int_0^L \varphi^2(x)\; dx = a^2 \int_0^L \frac{ \text{cn}^2 \left( \frac{4 \text{K}(k)}{L} x, k \right)}{ 1 - q \, \text{sn}^2 \left( \frac{4 \text{K}(k)}{L} x, k \right)} \; dx = \frac{a^2 L}{\text{K}(k)} \int_0^{\text{K}(k)} \frac{ \text{cn}^2(x,k)}{1-q \, \text{sn}^2(x,k)} \;dx.
    \end{equation*}
    Computing this integral via Maple and performing some simplifications, we find
    \begin{equation*}
        \|\varphi\|_{L^2}^2 = \frac{a^2 L}{\text{K}(k)} \left( \frac{\Pi(q,k) (q-1) + \text{K}(k)}{q} \right),
    \end{equation*}
    where $\Pi$ is the complete elliptic integral of the third kind defined by 
    \begin{equation*}
    \Pi(q, k) = \int_{0}^{\frac{\pi}{2}} \frac{d\theta}{(1 - q \sin^2 \theta) \sqrt{1 - k^2 \sin^2 \theta}}.
    \end{equation*}
    By substituting the expressions for $a$ and $q$ from \eqref{avaleuCnoidal} and \eqref{qvaleuCnoidal}, respectively, and employing Maple for further symbolic manipulation, we obtain
    \begin{equation}\label{ddkphi}
        \frac{d}{dk} \|\varphi\|_{L^2}^2 = \frac{\text{a}(k) \text{K}(k) + \text{b}(k) \text{E}(k)}{\text{r}(k) \sqrt{2-k^2+2\text{r}(k)} (k^2 - 1 - \text{r}(k))^2 k (1+\text{r}(k))}
    \end{equation}
    where $\text{r}(k) = \sqrt{k^4 - k^2 + 1}$, and the expressions $\text{a}(k)$ and $\text{b}(k)$ are given by
    \begin{equation*}
        \text{a}(k) = (-12k^6 + 42k^4 - 56k^2 + 32) \text{r}(k) + (12k^8 - 48k^6 + 82k^4 - 72k^2 + 32)
    \end{equation*}
    and
    \begin{equation*}
        \text{b}(k) = (-18k^4 + 40k^2 - 32) \text{r}(k) + (18k^6 - 50k^4 + 56k^2 - 32),
    \end{equation*}
    respectively. It is clear that the denominator of \eqref{ddkphi} is positive for all $k \in (0,1)$. On the other hand, the numerator $\mathcal{N}(k) = \text{a}(k) \text{K}(k) + \text{b}(k) \text{E}(k)$ satisfies
    \begin{equation*}
        \mathcal{N}(k) = \text{a}(k) \int_0^{\frac{\pi}{2}} \frac{dt}{\sqrt{ 1 - k^2 \sin^2(t)}} + \text{b}(k) \int_0^{\frac{\pi}{2}} \sqrt{ 1 - k^2 \sin^2(t)} \, dt = \int_0^{\frac{\pi}{2}} \frac{ \text{a}(k) + \text{b}(k) - \text{b}(k) k^2 \sin^2(t)}{ \sqrt{ 1 - k^2 \sin^2(t)}} \, dt,
    \end{equation*}
    which implies
    \begin{equation*}
        \mathcal{N}(k) \geq \int_0^{\frac{\pi}{2}} \left( \text{a}(k) + \text{b}(k) - \text{b}(k) k^2 \sin^2(t) \right) \, dt
    \end{equation*}
    for $t \in (0, \frac{\pi}{2})$ and $k \in (0,1)$. Thus, we obtain
    \begin{equation*}
        \mathcal{N}(k) \geq \left( \text{a}(k) + \text{b}(k) \right) \frac{\pi}{2} - \text{b}(k) k^2 \frac{\pi}{4} = \frac{\pi}{4} \underbrace{\left[ 2 \left( \text{a}(k) + \text{b}(k) \right) - \text{b}(k) k^2 \right]}_{:=\mathcal{M}(k)}.
    \end{equation*}
{Since $\mathcal{M}$ is a polynomial, it is easy to show that $\mathcal{M}$ is an increasing and positive function on $(0,1)$. Moreover, using Maple, we can verify that $\mathcal{M}(k)=0$ if and only if $k=0$. Thus, we also conclude that $\mathcal{N}(k) > 0$ for all $k \in (0,1)$, which implies $\frac{d}{dk} \|\varphi\|_{L^2}^2 > 0$ for all $k \in (0,1)$.}

    To establish that $\frac{d\omega}{dk}>0$ for all $k \in (0,1)$, we apply a similar argument. From the expression for $\omega$ in \eqref{valuewCnoidal}, we have
    \begin{equation*}
        \frac{d\omega}{dk} = \frac{16 \text{K}(k)}{L^2 (1-k^2) \, k \, \text{r}(k)} \underbrace{\left[ (2k^4 - 2k^2 + 2) \text{E}(k) + (- k^4 + 3k^2 - 2) \text{K}(k) \right]}_{:=\mathcal{P}(k)}.
    \end{equation*}
    We must show that $\mathcal{P}(k)$ is positive for all $k \in (0,1)$. Proceeding as before, we observe that
    \begin{equation*}
        \mathcal{P}(k) \geq \frac{\pi}{4} \left( 2k^4 + 2k^2 - (2k^4 - 2k^2 + 2)k^2 \right).
    \end{equation*}
{A direct verification of the polynomial on the right side of the inequality above shows that it is positive for all $k \in (0,1)$. Hence,  $\mathcal{P}(k) > 0$ for all $k \in (0,1)$, which yields $\frac{d\omega}{dk} > 0$ for all $k \in (0,1)$.  }    

    Therefore, we conclude that $\frac{d}{d\omega} \|\varphi\|_{L^2}^2 > 0$ for all $k \in (0,1)$, as desired.
\end{proof}

\begin{remark}\label{rem:correction0}
    We provide an algebraic proof of the fact that $\frac{d}{d\omega} \|\varphi\|_{L^2}^2 > 0$ for all $\omega \in \big( \tfrac{4\pi^2}{L^2}, +\infty \big)$ in Lemma \ref{dif-norma}, in contrast to many studies in the literature whose proof is based on a graphical analysis of this quantity's behavior, as seen in \cite[Section 6.1]{BittencourtLoreno2022} and  \cite[Subsection 4.3]{NataliCardosoAmaral}. 
\end{remark}

\begin{remark}\label{rem:correction}
    It is important to clarify that in a previous work \cite[Subsection 6.1]{BittencourtLoreno2022}, an implementation error occurred regarding the Heumann Lambda function while computing $\frac{d}{dk} \|\varphi\|_{L^2}^2$. We have now corrected this inaccuracy and verified that the results presented in Lemma \ref{dif-norma} are fully consistent with the calculations performed via the Heumann Lambda function, employing the formula (410.03) or the alternative formula (411.03) as found in \cite[pages 225-226]{byrd}.
\end{remark}

\begin{lemma}\label{L1-cnoidal}
Let $L>0$ and $\omega \in \left( \frac{4\pi^2}{L^2}, +\infty\right)$ be fixed. If  $\varphi=\varphi_\omega$ is the periodic solution with cnoidal profile given in Proposition \ref{cnoidalcurve}, then 
\begin{equation}
    \left( \mathcal{L}_1^{-1}(\varphi), \varphi \right)_{L^2} < 0.
\end{equation}
\end{lemma}

\begin{proof}
    Recall that $\omega \in \left( \frac{4\pi^2}{L^2}, +\infty\right) \longmapsto \varphi$ is a smooth curve of cnoidal solutions, as established in Proposition \ref{cnoidalcurve}. By differentiating both sides of \eqref{mainEDOquinticsystem} with respect to $\omega$, we find that $\mathcal{L}_1\left(\frac{d\varphi}{d\omega}\right) = -\varphi$, which implies $\mathcal{L}_1^{-1}(\varphi) = -\frac{d}{d\omega}\varphi$. Consequently, it follows that
    \begin{equation}\label{inverseL1dnoidalGeneral}
        (\mathcal{L}_1^{-1}(\varphi), \varphi)_{L^2} = - \int_0^L \varphi \frac{d\varphi}{d\omega} dx = - \frac{1}{2}\frac{d}{d\omega} \int_0^L \varphi^2(x) dx = -\frac{1}{2}\frac{d}{d\omega}\|\varphi\|_{L^2}^2.
    \end{equation}
    The desired result then follows directly from Lemma \ref{dif-norma}, which ensures that $\frac{d}{d\omega}\|\varphi\|_{L^2}^2>0$.
\end{proof}

\begin{lemma}\label{L2-cnoidal}
Let $L>0$ and $\omega \in \left( \frac{4\pi^2}{L^2},+\infty\right)$ be fixed. If $\varphi=\varphi_\omega$ is the periodic wave with cnoidal profile given in Proposition $\ref{cnoidalcurve}$, then
	\begin{equation}
		( \mathcal{L}_2^{-1} (\varphi'),\varphi' )_{L^2} > 0.
	\end{equation}
\end{lemma}
\begin{proof}
Recall that the operator $\mathcal{L}_2$ is self-adjoint in $L^2_{\text{per}}$, and that its eigenfunctions associated with distinct eigenvalues are orthogonal (see \cite[Theorem 11.21]{AnguloBook} and \cite[Lemma 3.4]{Schmudgen}). Also, we know that ${\rm n}(\mathcal{L}_2)=1$ and ${\rm z}(\mathcal{L}_2)=1$. These facts imply that there exists $f_0 \in {\rm D}(\mathcal{L}_2)$ and $-\lambda_0^{(2)} < 0$ such that $\mathcal{L}_2(f_0) = -\lambda_0^{(2)} f_0$, with ${\rm Ker}(\mathcal{L}_2) = [\varphi]$. On the other hand, since $\varphi' \in [\varphi]^{\perp} = {\rm Ker}(\mathcal{L}_2)^{\perp} = {\rm R}(\mathcal{L}_2)$, there exists\footnote{Since $L^2_{\text{per}} = {\rm Ker}(\mathcal{L}_2) \oplus {\rm Ker}(\mathcal{L}_2)^{\perp}$, there exist $g_2 \in {\rm D}(\mathcal{L}_2)$, $g_0 \in {\rm Ker}(\mathcal{L}_2)^\perp$ and $g_1 \in {\rm Ker}(\mathcal{L}_2)$ such that $\varphi' = \mathcal{L}_2(g_2) = \mathcal{L}_2(g_0) + \mathcal{L}_2(g_1) = \mathcal{L}_2(g_0)$.} $g_0 \in {\rm Ker}(\mathcal{L}_2)^{\perp}$, $g_0 \neq 0$, such that $\mathcal{L}_2(g_0) = \varphi'$, or equivalently $\mathcal{L}_2^{-1}(\varphi') = g_0$. Moreover, $\varphi'$ is an odd function, while $f_0$ is an even function (according to \cite[Theorem 1.1]{MagnusWinkler}, the eigenfunction associated with the first eigenvalue of $\mathcal{L}_2$ is even). Thus,
\[
(f_0, \varphi')_{L^2} = 0.
\]
Now, using the fact that $\mathcal{L}_2$ is self-adjoint, we obtain
\[
(f_0, \varphi')_{L^2} = (f_0, \mathcal{L}_2(g_0))_{L^2} = (\mathcal{L}_2(f_0), g_0)_{L^2}=-\lambda_0^{(2)}(f_0, g_0)_{L^2}.
\]
Since $(f_0, \varphi')_{L^2} = 0$ and $-\lambda_0^{(2)} \neq 0$, it follows that
\[
(f_0, g_0)_{L^2} = 0.
\]
Therefore, we conclude that $f_0$ and $g_0$ are orthogonal in $L^2_{\text{per}}$. In other words, the eigenfunction $f_0$ associated with the first negative eigenvalue of $\mathcal{L}_2$ is orthogonal to $g_0 = \mathcal{L}_2^{-1}(\varphi')$.

Next, we claim that $g_0$ and $\varphi$ are orthogonal in $L^2_{\text{per}}$. To prove this claim, we first show that $g_0$ is an odd function. To do so, considering the orthogonal decomposition of the space $L^2_{\text{per}}$ into subspaces of odd and even functions, we can write $g_0 = g_0^{\text{even}} + g_0^{\text{odd}}$, where $g_0^{\text{even}} \in L^2_{\text{per}}$ is an even function and $g_0^{\text{odd}} \in L^2_{\text{per}}$ is an odd function. Since $\mathcal{L}_2$ is a parity-preserving operator, the function $\mathcal{L}_2(g_0^{\text{even}})$ is even and $\mathcal{L}_2(g_0^{\text{odd}})$ is odd. Moreover, since $\mathcal{L}_2(g_0^{\text{even}}) + \mathcal{L}_2(g_0^{\text{odd}}) = \mathcal{L}_2(g_0) = \varphi'$, with $\varphi'$ an odd function, we obtain that $\mathcal{L}_2(g_0^{\text{even}}) = 0$, which implies $g_0^{\text{even}} = \alpha \varphi$ for some $\alpha \in \mathbb{R}$. However, we note that
\[
(g_0, \varphi)_{L^2} = (g_0^{\text{even}}, \varphi)_{L^2} + (g_0^{\text{odd}}, \varphi)_{L^2} \implies (g_0, \varphi)_{L^2} = (g_0^{\text{even}}, \varphi)_{L^2},
\]
due to $(g_0^{\text{odd}}, \varphi)_{L^2} = 0$. But, in view of $g_0 \in {\rm Ker}(\mathcal{L}_2)^{\perp} = [\varphi]^{\perp}$, it follows that $(g_0, \varphi)_{L^2} = 0$, and
\[
0 = (g_0^{\text{even}}, \varphi)_{L^2} = (\alpha \varphi, \varphi)_{L^2} = \alpha \|\varphi\|_{L^2}^2 \implies \alpha = 0 \quad \text{and} \quad g_0^{\text{even}} = 0.
\]
Thus, $g_0 = g_0^{\text{odd}}$ is an odd function. As a consequence, since $\varphi$ is an even function, we conclude that $g_0$ and $\varphi$ are orthogonal, as claimed.

Now, according to \cite[page 96]{AnguloBook}, the following orthogonal decomposition of $L^2_{\text{per}}$ holds:
\begin{equation}\label{L2decomposition}
L^2_{\text{per}} = [f_0] \oplus [\varphi] \oplus P,
\end{equation}
where $P \subset L^2_{\text{per}}$ is such that $P_c = P \cap {\rm D}(\mathcal{L}_2)$ is called the positive subspace of $\mathcal{L}_2$, that is, there exists $\varepsilon_0 > 0$ such that
\[
(\mathcal{L}_2(f), f)_{L^2} \geq \varepsilon_0 \|f\|_{L^2}^2, \quad \text{for all } f \in P_c.
\]
Since $g_0 \perp f_0$ and $g_0 \perp \varphi$, we conclude from the decomposition \eqref{L2decomposition} that $g_0 \in P_c$. Consequently, there exists $\varepsilon > 0$ such that
\[
(\mathcal{L}_2^{-1}(\varphi'), \varphi')_{L^2} = (\mathcal{L}_2(g_0), g_0)_{L^2} \geq \varepsilon \|g_0\|_{L^2}^2 > 0.
\]
This completes the proof of the lemma.
\end{proof}

\begin{remark}
We can also perform a numerical approach as is \cite[Section 3]{NataliCardosoAmaral} (see also \cite{BittencourtLoreno2022} and \cite{NataliMoraesLorenoPastor}) to conclude that $(\mathcal{L}_2^{-1}(\varphi'), \varphi')_{L^2}>0$. 
\end{remark}

Having established the previous lemmas, we are now positioned to present the main results of this section: the spectral stability of the multiple periodic solutions $(u,v) = (\varphi, B\varphi)$ in the space $\mathbb{H}^1_{\text{per}}$. To this end, since the spectral analysis of the operator $\mathcal{L}$ depends on $B$ given in \eqref{Bvalue}, we shall examine the spectral stability scenarios depending of $B$ separately:

\subsection*{Case 1: $B = B_+$} 
In this case, we find that ${\rm z}(\mathcal{L}) = 2$, where ${\rm Ker}(\mathcal{L}) = [\Theta_1, \Theta_2]$ is given by
\begin{equation*}
    \Theta_1 = (\varphi', B\varphi', 0, 0) \quad \text{and} \quad \Theta_2 = (0, 0, \varphi, B\varphi).
\end{equation*}
By employing \eqref{VentriesExplicity}, the matrix $J$ defined in \eqref{Jmatrix}, and the orthogonal matrix $U$ (computed via Maple for the case $B = B_+$), it follows that
$$
U^{-1}J \Theta_1 = \left( 0,0, \frac{1}{a_3}\;2\gamma^2 \varphi',0  \right)  \quad \text{and} \quad  U^{-1}J \Theta_2 =\left( - \frac{1}{a_1 }\frac{1}{\kappa^2}\; \varphi,0,0,0  \right).
$$
where $a_1,a_3>0$ are appropriate constants arising from the normalization of the eigenvectors that form the columns of $U$.

Consequently, the matrix $V$ is expressed as
\begin{equation*}
    V = \begin{bmatrix}
        \left( \frac{1}{a_3}\right)^2  4 \gamma^4 \left(  \mathcal{L}_2^{-1}(\varphi'), \varphi'\right)_{L^2} & 0 \\
    0 &    \left( \frac{1}{a_1}\right)^2 \frac{1}{\kappa^4} \left(\mathcal{L}_1^{-1} (\varphi), \varphi\right)_{L^2}
    \end{bmatrix}.
\end{equation*}
 
Applying Lemmas \ref{L1-cnoidal} and \ref{L2-cnoidal}, we establish that ${\rm n}(V) = 1$. Since ${\rm n}(\mathcal{L}) = 5$, it follows that $\mathcal{K}_{\rm Ham} = 4$. Given that $\mathcal{K}_{\rm Ham}$ is a non-zero even integer, no conclusion can be drawn regarding spectral stability in this case.

\subsection*{Case 2: $B = B_-$} 
In this case, the kernel of $\mathcal{L}$ remains the same as in the previous case. However, following analogous computations with the matrices $J$ and $U$ specific to this setting, we obtain
$$
U^{-1} J \Theta_1=\left( 0,0, -\frac{1}{b_3}\;\frac {1}{ \left( \gamma\,\sqrt {{\gamma}^{2}-4\,{\kappa}^{2}}+2
\,{\kappa}^{2}-{\gamma}^{2} \right) ^{2}}\,\varphi',0  \right) \quad \text{and} \quad U^{-1}J \Theta_2 =\left( - \frac{1}{b_1 }\frac{1}{\kappa^2}\; \varphi,0,0,0  \right)
$$
where $b_1,b_3>0$ are suitable constants that arise from the normalization of the eigenvectors that form the columns of $U$.

By proceeding as in \textbf{Case 1}, we find that ${\rm n}(V) = 1$. Since ${\rm n}(\mathcal{L}) = 4$, it follows that $\mathcal{K}_{\rm Ham} = 4$. Consequently, as $\mathcal{K}_{\rm Ham}$ is a non-zero even integer, this case also yields no conclusive result regarding spectral stability.

\subsection*{Case 3: $B = 1$} 

In this case, we separate our analysis into four partitions of $\gamma$, following the same procedure as in the previous spectral analysis.

\subsubsection*{Case 3.1: $\gamma = 0$} 
For $\gamma = 0$, we find that ${\rm z}(\mathcal{L}) = 4$, with the kernel of $\mathcal{L}$ given by
\begin{equation*}
    {\rm Ker}(\mathcal{L}) = [(\varphi',0,0,0), (0,\varphi',0,0), (0,0,\varphi,0), (0,0,0,\varphi)].
\end{equation*}
Furthermore, the use of the matrix $U$ is unnecessary here since $\mathcal{L}$ assumes the block diagonal form
\begin{equation*}
    \mathcal{L} = \begin{bmatrix}

        \mathcal{L}_1 & 0 & 0 & 0 \\

        0 & \mathcal{L}_1 & 0 & 0 \\

        0 & 0 & \mathcal{L}_2 & 0 \\

        0 & 0 & 0 & \mathcal{L}_2

    \end{bmatrix}.
\end{equation*}
By applying the matrix $J$ as defined in \eqref{Jmatrix}, we obtain
\begin{equation*}
    V = \begin{bmatrix}
        \left(\mathcal{L}_2^{-1} (\varphi'), \varphi'\right)_{L^2} & 0 & 0 & 0 \\
        0 & \left(\mathcal{L}_2^{-1} (\varphi'), \varphi'\right)_{L^2} & 0 & 0 \\
        0 & 0 & \left(\mathcal{L}_1^{-1} (\varphi), \varphi\right)_{L^2} & 0 \\
        0 & 0 & 0 & \left(\mathcal{L}_1^{-1} (\varphi), \varphi\right)_{L^2}
    \end{bmatrix}.
\end{equation*}
From Lemmas \ref{L1-cnoidal} and \ref{L2-cnoidal}, it follows that ${\rm n}(V) = 2$. Given that ${\rm n}(\mathcal{L}) = 6$, we conclude that $\mathcal{K}_{\rm Ham} = 4$.

\subsubsection*{Case 3.2: $\gamma \in (0,2\kappa)$} 
In this case, we have ${\rm z}(\mathcal{L}) = 2$ with $\Theta_1 = (\varphi', B \varphi', 0,0)$ and $\Theta_2 = (0,0, \varphi, B \varphi)$ for all $\gamma \in \left( 0, 2\kappa \right)$. By employing the expressions for $J$ and $U$, we obtain
\begin{equation*}
    U^{-1} J \Theta_1 = (0,0, -\sqrt{2} \varphi', 0) \quad \text{and} \quad U^{-1} J \Theta_2 = (\sqrt{2} \varphi, 0, 0, 0),
\end{equation*}
which implies
\begin{equation*}
    V = \begin{bmatrix}
        2 \left(\mathcal{L}_2^{-1} (\varphi'), \varphi'\right)_{L^2} & 0 \\
        0 & 2 \left(\mathcal{L}_1^{-1} (\varphi), \varphi\right)_{L^2}
    \end{bmatrix}.
\end{equation*}
Invoking Lemmas \ref{L1-cnoidal} and \ref{L2-cnoidal}, we find that ${\rm n}(V) = 1$. Since ${\rm n}(\mathcal{L}) = 5$ in this range, we again obtain $\mathcal{K}_{\rm Ham} = 4$.

\subsubsection*{Case 3.3: $\gamma = 2\kappa$} 

For $\gamma = 2 \kappa$, we have ${\rm z}(\mathcal{L}) = 3$, where the kernel is given by
\begin{equation*}
    {\rm Ker}(\mathcal{L}) = [(\varphi',B\varphi',0,0), (0,0,\varphi,B\varphi), (-B\varphi,\varphi,0,0)].
\end{equation*}
From the explicit expressions of $J$ and $U$, it follows that $U^{-1} J \Theta_1 = (0,0, -\sqrt{2} \varphi', 0)$, $U^{-1} J \Theta_2 = (\sqrt{2} \varphi, 0, 0, 0)$, and $U^{-1} J \Theta_3 = (0,0,0,\sqrt{2}\varphi)$. Consequently, the matrix $V$ is given by
\begin{equation*}
    V = \begin{bmatrix}
        2 \left(\mathcal{L}_2^{-1} (\varphi'), \varphi'\right)_{L^2} & 0 & 0 \\
        0 & 2 \left(\mathcal{L}_1^{-1} (\varphi), \varphi\right)_{L^2} & 0 \\
        0 & 0 & 2 \left(\mathcal{L}_4^{-1} (\varphi), \varphi\right)_{L^2}
    \end{bmatrix}.
\end{equation*}
Since ${\rm n}(\mathcal{L}_4) = {\rm z}(\mathcal{L}_4) = 0$, Lemma \ref{positivity} implies that $\left(\mathcal{L}_4^{-1} (\varphi), \varphi\right)_{L^2} > 0$. Thus, by invoking Lemmas \ref{L1-cnoidal} and \ref{L2-cnoidal}, we find that ${\rm n}(V) = 1$. Given that ${\rm n}(\mathcal{L}) = 4$, we obtain $\mathcal{K}_{\rm Ham} = 3$. Therefore, since $\mathcal{K}_{\rm Ham}$ is a non-zero odd integer, we conclude that the multiple periodic wave $(\varphi, B \varphi)$ is spectrally unstable in $\mathbb{H}^1_{\text{per}}$ for $\gamma = 2\kappa$.

\subsubsection*{Case 3.4: $\gamma \in (2\kappa, +\infty)$}
In this case, it can be shown that ${\rm n}(\mathcal{L}) = 3$ and ${\rm z}(\mathcal{L}) = 2$ for $\gamma \in (5\kappa, +\infty)$. This leads to the same scenario encountered in Case 3.2, where ${\rm n}(V) = 1$. Consequently, we find that $\mathcal{K}_{\rm Ham} = 2$. Since $\mathcal{K}_{\rm Ham}$ is a non-zero even integer, no definitive conclusion can be drawn regarding spectral stability in this case.

\subsection*{Case 4: $B = -1$} 
Here, the spectral analysis has been developed only for $\gamma = 0$. Under this condition, we obtain the same result as in Case 3.1, where ${\rm n}(\mathcal{L}) = 6$ and ${\rm n}(V) = 2$. It follows that $\mathcal{K}_{\rm Ham} = 4$. Since $\mathcal{K}_{\rm Ham}$ is a non-zero even integer, no conclusion can be drawn regarding spectral stability in this case.

Summarizing the analysis for each case of $B$, a spectral instability result for the multiple periodic waves $(\varphi, B\varphi, 0, 0)$ in $\mathbb{H}^1_{\text{per}}$ is established only for $B=1$ and $\gamma = 2\kappa$. More precisely, we have the proof of Theorem \ref{teorema-cnoidal}.

\begin{remark}\label{rmk1}
    All other values of $B$ in  \eqref{Bvalue} and the remaining ranges of $\gamma$ for $B = 1$ remain open problems in the literature. This is due to the high number of negative eigenvalues of $\mathcal{L}$ and the fact that the Krein Hamiltonian Index $\mathcal{K}_{\rm Ham}$ yields inconclusive even values in those cases.
\end{remark}

\vspace{0.5cm}
\section{Spectral Stability on subspace of odd functions}\label{spectral-stability-odd-space}
Motivated by Remark \ref{rmk1}, we shall investigate the spectral stability of $(\psi, B\psi, 0, 0)$ in $\mathbb{H}^1_{\text{per,odd}}$, where $\psi$ is an odd sign-changing solution with the snoidal profile given in \eqref{snoidal-solution}. Recalling that $\psi(x) = \varphi\left( x + \tfrac{L}{4} \right)$ for all $x \in \mathbb{R}$, we can establish certain properties regarding the operators $\mathcal{L}_i$, $i=1,2,3,4$, which are enunciated in the following lemmas:

\begin{lemma}\label{L1-snoidal}
Let $L>0$ and $\omega \in \left( \frac{4\pi^2}{L^2}, +\infty\right)$ be fixed. If $\psi=\psi_\omega$ is the periodic wave with the snoidal profile given in Proposition \ref{snoidalcurve}, then 
\begin{equation}
    \left( \mathcal{L}_{1,\rm odd}^{-1}(\psi), \psi \right)_{L^2} < 0.
\end{equation}
\end{lemma}
\begin{proof}
    Analogous to the proof of Lemma \ref{L1-cnoidal}, we observe that $\mathcal{L}_1\left( \frac{d \psi}{d \omega} \right) = - \psi$, which implies $\mathcal{L}_1^{-1}(\psi) = - \frac{d\psi}{d\omega}$. Consequently, using the fact that $\mathcal{L}_{1,\rm odd}$ is the restriction of $\mathcal{L}_1$ on the subspace of odd functions, we have that
    \begin{equation*}
        \left( \mathcal{L}_{1,\rm odd}^{-1}(\psi), \psi \right)_{L^2} = \left( \mathcal{L}_1^{-1}(\psi), \psi \right)_{L^2} = \left( -\frac{d\psi}{d\omega}, \psi \right)_{L^2} = - \frac{1}{2} \frac{d}{d\omega} \int_0^L \psi^2(x) dx.
    \end{equation*}

    Since $\psi$ is a translation of $\varphi$ by $\frac{L}{4}$, and by virtue of the $L$-periodicity of $\varphi$, it follows that
    \begin{equation*}
    \int_0^L \psi^2(x)\; dx = \int_0^L \varphi^2\left( x + \frac{L}{4} \right) \; dx = \int_{\frac{L}{4}}^{\frac{5L}{4}} \varphi^2(y)\; dy = \int_{0}^L \varphi^2(x) \; dx.   
    \end{equation*}
    Thus, we obtain $\left( \mathcal{L}_{\rm 1,{\rm odd}}^{-1}(\psi), \psi \right)_{L^2} = -\frac{1}{2} \frac{d}{d\omega} \|\varphi\|_{L^2}^2$. The result then follows immediately from Lemma \ref{dif-norma}.
\end{proof}

\begin{lemma}\label{L2-snoidal}
Let $L>0$ and $\omega \in \left( \frac{4\pi^2}{L^2}, +\infty\right)$ be fixed. If $\psi=\psi_\omega$ is the periodic wave with the snoidal profile given in Proposition \ref{snoidalcurve}, then 
\begin{equation}
    \left( \mathcal{L}_{\rm 2,{\rm odd}}^{-1}(\psi'), \psi' \right)_{L^2} > 0.
\end{equation}
\end{lemma}
\begin{proof}
    The result follows directly from the fact that ${\rm n}(\mathcal{L}_{\rm 2,odd}) = 0$ and ${\rm z}(\mathcal{L}_{\rm 2,odd}) = 1$ in $\mathbb{H}^2_{\text{per,odd}}$, together with the application of Lemma \ref{positivity}.
\end{proof}

Taking into account the reduction of the number of negative eigenvalues and the dimension of the kernel of $\mathcal{L}$ when restricted to $\mathbb{H}^2_{\text{per,odd}}$, we are now positioned to establish spectral stability results for the multiple periodic waves $(\psi, B\psi, 0, 0)$ through the various scenarios for $B$. Following our previous approach, we shall partition this analysis according to the cases for $B$ defined in \eqref{Bvalue}.

\subsection*{Case 1: $B = B_+$} In this scenario, we established that ${\rm n}(\mathcal{L}_{\text{odd}}) = 2$ and ${\rm Ker}(\mathcal{L}_{\text{odd}}) = [(0,0, \psi, B\psi)]$. It follows that $V$ is a one-dimensional matrix defined by 
\[
V = \frac{1}{\kappa^4} \left( \mathcal{L}_{1,\text{odd}}^{-1} (\psi), \psi \right)_{L^2},
\]
given that $U^{-1} J \Theta_1 = \left( -\frac{1}{\kappa^2} \psi, 0,0,0 \right)$. According to Lemma \ref{L1-snoidal}, we have ${\rm n}(V) = 1$. Consequently, we find that $\mathcal{K}_{\text{Ham}} = 1$, which implies that the multiple periodic wave $(\psi, B\psi, 0, 0)$ is spectrally unstable for the case $B = B_+$.

\subsection*{Case 2: $B = B_-$} Since ${\rm n}(\mathcal{L}_{\text{odd}}) = 2$ and ${\rm Ker}(\mathcal{L}_{\text{odd}}) = [(0,0, \psi, B\psi)]$, this case mirrors the previous one. Thus, we conclude that $(\psi, B\psi, 0, 0)$ is spectrally unstable when $B = B_-$.

\subsection*{Case 3: $B = 1$} Here, we examine the four distinct scenarios corresponding to the parameter $\gamma$.

\subsubsection*{Case 3.1: $\gamma = 0$} As previously established, we have ${\rm Ker}(\mathcal{L}_{\rm odd}) = [(0,0,\psi,0), (0,0,0,\psi)]$. Consequently, the matrix $V$ is defined as
\begin{equation*}
    V = \begin{bmatrix}
        \left( \mathcal{L}_{\rm 1,odd}^{-1} (\psi), \psi \right)_{L^2} & 0 \\
        0 & \left( \mathcal{L}_{\rm 1,odd}^{-1} (\psi), \psi \right)_{L^2}
    \end{bmatrix}.
\end{equation*}
In view of Lemma \ref{L1-snoidal}, it follows that ${\rm n}(V) = 2$. Given that ${\rm n}(\mathcal{L}_{\rm odd}) = 2$ for $B=1$ and $\gamma = 0$, we obtain $\mathcal{K}_{\text{Ham}} = 0$. This result implies that the multiple periodic wave $(\psi, B\psi, 0, 0)$ is spectrally stable for the specific case where $B = 1$ and $\gamma = 0$.

\subsubsection*{Case 3.2: $\gamma \in (0,2\kappa)$} Now, we have that ${\rm Ker}(\mathcal{L}_{\rm odd}) = [(0,0,\psi, B\psi)]$, implying that $V = 2 \left( \mathcal{L}_{\rm 1,odd}^{-1} (\psi), \psi \right)_{L^2}$ and, consequently, using the Lemma \ref{L1-snoidal}, ${\rm n}(V) = 1$. On the other hand, by the fact that ${\rm n}(\mathcal{L}_{\rm odd}) = 2$, we obtain that $\mathcal{K}_{\text{Ham}} = 1$ implying that the multiple periodic wave $(\psi, B\psi, 0, 0)$ is spectrally unstable in this case.

\subsubsection*{Case 3.3: $\gamma = 2\kappa$} As estabilished in the section of spectral analysis, we have that ${\rm Ker}(\mathcal{L}_{\rm odd}) = [(0,0,\psi,B\psi), (-B\psi,\psi,0,0)]$. Thus, using the expressions of matrices $J$ and $U$, we have that the matrix $V$ is defined as
\begin{equation*}
    V = \begin{bmatrix}
        2 \left( \mathcal{L}_{\rm 1,odd}^{-1} (\psi), \psi \right)_{L^2} & 0 \\
        0 & 2 \left( \mathcal{L}_{\rm 4,odd}^{-1} (\psi), \psi \right)_{L^2}
    \end{bmatrix}.
\end{equation*}
From the spectral analysis of $\mathcal{L}_{\rm 4,odd}$ and Lemma \ref{positivity}, we have that $\left( \mathcal{L}_{\rm 4,odd}^{-1} (\psi), \psi \right)_{L^2} > 0$. Thus, in view of Lemma \ref{L1-snoidal}, it follows that ${\rm n}(V) = 1$. Given that ${\rm n}(\mathcal{L}_{\rm odd}) = 1$ for $B=1$ and $\gamma = 2\kappa$, we obtain $\mathcal{K}_{\text{Ham}} = 0$. This result implies that the multiple periodic wave $(\psi, B\psi, 0, 0)$ is spectrally stable for the particular case where $B = 1$ and $\gamma = 2\kappa$.

\subsubsection*{Case 3.4: $\gamma \in (2\kappa, +\infty)$} Since ${\rm Ker}(\mathcal{L}_{\text{odd}}) = [(0,0, \psi, B\psi)]$, this case mirrors the Case 3.2 and we have that ${\rm n}(V) = 1$. However, in this case, we have that ${\rm n}(\mathcal{L}_{\rm odd}) = 1$ implying that $\mathcal{K}_{\text{Ham}} = 0$. Therefore, we conclude that the multiple periodic wave $(\psi, B\psi, 0, 0)$ is spectrally stable considering $B = 1$ and $\gamma \in (2\kappa, +\infty)$.

\subsection*{Case 4: $B=-1$} In this scenario, spectral analysis results are available only for the case $\gamma = 0$. Furthermore, this situation is entirely analogous to Case 3.1, where it was established that ${\rm n}(V) = 2$. Since ${\rm n}(\mathcal{L}_{\rm odd}) = 2$, we obtain $\mathcal{K}_{\text{Ham}} = 0$. Consequently, the multiple periodic wave $(\psi, B\psi, 0, 0)$ is spectrally stable for $B = -1$ and $\gamma = 0$.

Summarizing the analysis for each case of $B$, a spectral stability/instability result for the multiple periodic waves $(\varphi, B\varphi, 0, 0)$ in $\mathbb{H}^1_{\text{per,odd}}$ is established. More precisely, we have the proof of Theorem \ref{teorema-snoidal}.

\section{Scalar NLS Equation: Orbital Stability on subspace of odd functions}\label{section-orbital}

First, let ${\rm d}: \left( \tfrac{4\pi^2}{L^2}, +\infty \right) \to \mathbb{R}$ be the action functional defined by ${\rm d}(\omega) = E(\varphi,0) + \omega F(\varphi,0)$, where $E$ and $F$ are given in \eqref{EnergyFunctional} and \eqref{MassFunctional}, respectively, obtained by setting $v = 0$ in \eqref{NLS-system}. Thus, we have that 
$$
{\rm d}''(\omega) = \frac{1}{2} \frac{d}{d\omega} \|\varphi\|_{L^2}^2 \text{ for all } \omega \in \left( \tfrac{4\pi^2}{L^2}, +\infty \right).
$$

In \cite{BittencourtLoreno2022}, the authors initially proposed the existence of a frequency threshold $\omega^*>0$ such that ${\rm d}''(\omega) < 0$ for all $\omega \in (\omega^*,\infty)$.  By assuming $\omega \in (\omega^*,\infty)$ and by considering ${\rm p}({\rm d}'')$ as the number of positive eigenvalues associated to the second derivative of the funcion ${\rm d}$ we would be able to conclude that  the difference ${\rm n}(\mathcal{L}_{\text{even}}) - {\rm p}({\rm d}'')$ is an odd integer number (actually $3$), which would imply the orbital instability of the standing wave $\Phi = (\varphi,0)$ in the subspace $H^1_{\text{per,even}}$ of $H^1_{\text{per}}$  constituted by even functions. Here $\mathcal{L}_{\text{even}}$ denote the linearized operator $\mathcal{L}$ restricted to the even periodic space.

However, based on the corrections detailed in Lemma \ref{dif-norma} and  Remark  \ref{rem:correction}, we have identified a technical oversight in \cite{BittencourtLoreno2022} regarding the sign of ${\rm d}''(\omega)$. More precisely, our updated analysis confirms that ${\rm d}''(\omega) > 0$ for all $\omega \in \left( \frac{4\pi^2}{L^2}, +\infty \right)$. In view of the classical stability theory applied to this corrected scenario, we conclude that ${\rm n}(\mathcal{L}_{\text{even}}) - {\rm p}({\rm d}'') = 2$, because  ${\rm n}(\mathcal{L}_{\text{even}}) = 3$ and ${\rm p}({\rm d}'') = 1$. Since, in whole space or restricted to the even subspace, this difference is an even integer, the orbital stability or instability of $(\varphi,0)$ remains inconclusive in $H^1_{\text{per}}$ or its even subspace.

Following the arguments developed in the present work, a definitive conclusion can be reached by replacing the analysis to the subspace of odd functions. In this setting, the linearized operator $\mathcal{L}_{\text{odd}}$ satisfies ${\rm n}(\mathcal{L}_{\text{odd}}) = 1$. Combined with the fact that ${\rm d}''(\omega) > 0$ for all $\omega > \frac{4\pi^2}{L^2}$, the stability criteria established by Grillakis, Shatah, and Strauss in \cite{grillakis1} are fully satisfied. Consequently, we conclude that the periodic standing wave with a snoidal profile $\Psi = (\psi, 0)$ is orbitally stable in the space $H^1_{\text{per,odd}}$.

\vspace{0.5cm}

\section*{Concluding Remarks}

In this work, we addressed the spectral stability of periodic waves $(u, v) = (\varphi, B\varphi)$, where $\varphi$ has cnoidal profile as given in \eqref{cnoidal-solution}. Throughout our analysis, we observed that the linearized operator $\mathcal{L}$ -- associated with the Hamiltonian system $U_t = J\mathcal{L}U$ -- possesses a high number of negative eigenvalues. This spectral setting typically leads to non-zero even values for the Hamiltonian Index $\mathcal{K}_{\text{Ham}}$, which significantly complicates the stability analysis in the full energy space.

To overcome these difficulties, we focused our study on the subspace of odd periodic functions $\mathbb{H}^1_{\text{per,odd}}$ by considering the snoidal profile $\psi(x) = \varphi(x + L/4)$, for all $x \in \mathbb{R}$. By restricting the problem to this subspace, we were able to establish precise spectral stability and instability results. It is worth noting that, in cases where spectral instability is found within the subspace of odd functions, such instability is expected to persist in the full periodic space $\mathbb{H}^1_{\text{per}}$.

A particularly interesting phenomenon occurs in the case $B = 1$ and $\gamma = 2\kappa$. Our results demonstrate that the snoidal wave is spectrally stable within $\mathbb{H}^1_{\text{per,odd}}$. However, we also prove that it is unstable in the full periodic space $\mathbb{H}^1_{\text{per}}$. This highlights that spectral stability in the subspace of odd functions does not, in general, allow for an extension to the entire space.

\vspace{0.5cm}

\end{document}